\documentclass{article}
\usepackage[english]{babel}

\usepackage[a4paper,top=2cm,bottom=2cm,left=3cm,right=3cm,marginparwidth=1.75cm]{geometry}

\usepackage{setspace}
\usepackage{amsmath}
\usepackage{amsthm}
\usepackage{amsfonts,amssymb}
\usepackage{mathtools}
\usepackage{lmodern}

\usepackage{extarrows}
\usepackage{tikz-cd}

\usepackage{booktabs}
\usepackage{tabularx}
\usepackage{multicol,multirow}
\usepackage{longtable}
\usepackage{makecell}
\usepackage{tabularray}

\usepackage{graphicx}
\usepackage[format=hang,font=small,textfont=it]{caption}
\usepackage{subcaption} 

\usepackage{enumerate}
\usepackage{indentfirst}

\usepackage{fancyhdr}
\usepackage{float}

\usepackage[nottoc]{tocbibind}
\usepackage{tocloft} 

\usepackage{xcolor}
\usepackage{listings}
\usepackage{comment}

\usepackage[colorlinks=true, allcolors=blue]{hyperref}

\usepackage{titlesec}
\titlelabel{\thetitle.~}

\makeatletter
\renewenvironment{proof}[1][\proofname]{\par
	\pushQED{\qed}%
	\normalfont \topsep6\p@\@plus6\p@\relax
	\trivlist
	\item[\hskip\labelsep
	\bfseries
	#1\@addpunct{.}]\ignorespaces
}{%
	\popQED\endtrivlist\@endpefalse
}
\makeatother

\theoremstyle{plain}
\newtheorem{thm}{Theorem}[section]
\newtheorem{prop}[thm]{Proposition}
\newtheorem{lemma}[thm]{Lemma}
\newtheorem{corollary}[thm]{Corollary}
\newtheorem{question}[thm]{Question}

\theoremstyle{definition}
\newtheorem{definition}[thm]{Definition}

\newtheorem{remark}[thm]{Remark}
\newtheorem{example}[thm]{Example}

\newtheorem{convention}[thm]{Convention}

\renewcommand\arraystretch{2}

\renewcommand\appendix{\par
	\setcounter{section}{0}
	\setcounter{subsection}{0}
	\gdef\thesection{\Alph{section}}}

\numberwithin{equation}{section}

\DeclareMathOperator{\id}{id}

\DeclareMathOperator{\im}{im}
\DeclareMathOperator{\gr}{gr}
\DeclareMathOperator{\Span}{Span}
\DeclareMathOperator{\lex}{lex}
\DeclareMathOperator{\supp}{supp}

\DeclareMathOperator{\GKdim}{GKdim}

\newcommand{\PP}{\mathcal{P}}
\newcommand{\GG}{\mathbb{G}}
\newcommand{\II}{\mathbb{I}}
\newcommand{\dd}{\mathfrak{d}}
\newcommand{\rr}{\mathfrak{r}}
\newcommand{\germ}{\mathfrak}
\makeatletter
\renewcommand*\env@matrix[1][\arraystretch]{%
	\edef\arraystretch{#1}%
	\hskip -\arraycolsep
	\let\@ifnextchar\new@ifnextchar
	\array{*\c@MaxMatrixCols c}}
\makeatother

\makeatletter
\newcommand{\subjclass}[2][1991]{%
	\let\@oldtitle\@title%
	\gdef\@title{\@oldtitle\footnotetext{#1 \emph{Mathematics Subject Classification.} #2}}%
}
\newcommand{\keywords}[1]{%
	\let\@@oldtitle\@title%
	\gdef\@title{\@@oldtitle\footnotetext{\emph{Key words and phrases.} #1.}}%
}
\makeatother

\title{The structure of iterated Hopf Ore extensions}
\author{B.-F. Yu, G.-S. Zhou}
\date{}

\newcommand{\Addresses}{{%
		\bigskip
		\footnotesize
		
		B.-F. Yu, \textsc{School of Mathematics and Statistics, Ningbo University, Ningbo 315211, China}\par\nopagebreak
		\textit{E-mail address}: \texttt{yubofan59@gmail.com}
		
		\medskip
		
		G.-S. Zhou, \textsc{School of Mathematics and Statistics, Ningbo University, Ningbo 315211, China}\par\nopagebreak
		\textit{E-mail address}: \texttt{zhouguisong@nbu.edu.cn}
}}


\begin{document}
	\maketitle

	\begin{abstract}
		This paper studies iterated Hopf Ore extensions (IHOEs) over a field $\Bbbk$ of characteristic zero, a significant class of connected Hopf algebras with finite Gelfand–Kirillov dimension. We  establish two fundamental structural properties for arbitrary IHOEs of $\Bbbk$. First, the class of IHOEs of $\Bbbk$ is closed under Hopf subalgebras and quotient Hopf algebras. Second, every one-sided coideal subalgebra of an IHOE of $\Bbbk$ is an iterated Ore extension of $\Bbbk$. These structural facts are built upon the thin replacement machinery for PBW generating systems originating from Kharchenko’s work. We further prove a general inheritance theorem for PBW generating systems, which serves as a systematic perturbation method to produce PBW generating systems for subalgebras from those of the ambient algebra under mild hypotheses. Based on these theoretical advances, we develop an explicit combinatorial algorithm for  classifying  all one-sided coideal subalgebras of arbitrary IHOEs of $\Bbbk$. As practical illustrations, we explicitly classify all right coideal subalgebras of noncocommutative connected Hopf algebras of GK-dimension three.
	\end{abstract}

	\section{Introduction}

	The classification of Hopf algebras of finite Gelfand–Kirillov (GK) dimension represents a central, long-standing problem in Hopf algebra theory. Two major research programs have been advanced to tackle this classification problem: the lifting theory for pointed Hopf algebras pioneered by Andruskiewitsch and Schneider \cite{AS02,AS10}, and the homological classification framework developed by Brown, Goodearl and Zhang \cite{BZ08,GZ10}. The majority of known classification theorems rely on restrictive hypotheses such as pointedness, smoothness or integrality, and exhaustive classifications remain essentially confined to Hopf algebras of GK-dimension at most two, whereas the higher-dimensional case remains widely uncharted. For further progress along these two directions, we refer the reader to the survey articles \cite{An25,AA17,BZ21} and the references cited therein.
	
	Considerable structural simplifications emerge when restricting attention to \emph{connected} Hopf algebras, a distinguished subclass within finite-GK-dimensional Hopf algebras. A Hopf algebra is called \emph{connected} if its coradical is one-dimensional. Standard examples include universal enveloping algebras of Lie algebras, coordinate rings of unipotent algebraic groups, and various combinatorial connected graded Hopf algebras, such as those of (quasi-)symmetric functions, permutations, and rooted trees. In particular, the latter family plays a foundational role in Connes and Kreimer’s rigorous formulation of quantum field renormalization \cite{CK98}. While commutative and cocommutative connected Hopf algebras were thoroughly characterized in the 1960s \cite{MM65,CP21}, systematic structural investigations of their noncommutative noncocommutative counterparts have only blossomed over the past decade; see, for instance, \cite{Zh13,WZZ15,Wa15,BOZZ15,BG16,BGZ19,ZSL20,BZ22,Zh24,HW26}.
	
	Hopf Ore extensions serve as fundamental building blocks for constructing and investigating noncommutative, noncocommutative Hopf algebras of large GK-dimension. Iterated Hopf Ore extensions, first introduced in \cite{BOZZ15}, constitute an iterative generalization of ordinary Hopf Ore extensions. By definition, an \emph{$n$-step Iterated Hopf Ore extension} (\emph{$n$-step IHOE}) of a Hopf algebra $R$ is a Hopf algebra $H$ admitting an ascending chain of Hopf subalgebras 
	\[
	R=H_{0} \subset H_{1} \subset \cdots \subset H_{n}=H,
	\]
	where each $H_i$ is an Ore extension of $H_{i-1}$ for all $1\le i\le n$. A fundamental result from \cite[Theorem 3.2]{BOZZ15} ensures that every $n$-step IHOE of $\Bbbk$ is a connected Hopf algebra with GK-dimension precisely $n$. This property enables the systematic construction of connected Hopf algebras of arbitrary finite GK-dimension. Low-dimensional classification results demonstrate the ubiquity of the IHOE structure: all connected Hopf algebras with GK-dimension $\le 4$ are IHOEs of $\Bbbk$ \cite{Zh13,WZZ15}. Hu extended this bound to GK-dimension $\le 6$ in her doctoral work \cite{Hu26},  without carrying out complete classification for such objects. Moreover, by \cite[Theorem B]{ZSL20}, every finite-GK-dimensional connected graded Hopf algebra is an IHOE of $\Bbbk$. Nonetheless, the IHOE structure does not capture all finite-GK-dimensional connected Hopf algebras. Classic counterexamples include universal enveloping algebras of finite-dimensional simple Lie algebras of dimension at least four, while more sophisticated noncocommutative counterexamples are constructed in \cite{HW26}.
	
	It is well known that finite-GK-dimensional connected Hopf algebras are closed under taking Hopf subalgebras and quotient Hopf algebras. In this work, we establish the analogous closure property for the class of IHOEs of $\Bbbk$, which forms our first main theorem.
	
	\begin{thm}[Theorem \ref{Hopf-subalgebra-structure}, Theorem \ref{Quotients-IHOE}]
		\label{main-Sub-Quotient}
		Every Hopf subalgebra and every quotient Hopf algebra of an IHOE of $\Bbbk$ is again an IHOE of $\Bbbk$.
	\end{thm}
	
	Combining \cite[Theorem B]{ZSL20} and Theorem \ref{main-Sub-Quotient}, it follows that every Hopf subalgebra of a finite-GK-dimensional connected graded Hopf algebra is an IHOE of $\Bbbk$ (Corollary \ref{coideal-structure-homo}). For homogeneous substructures, a stronger relative result was established by Li and the second-named author in \cite[Theorem B]{LZ23}: every finite-GK-dimensional connected graded Hopf algebra is an IHOE of each of its homogeneous Hopf subalgebras. These results raise a natural problem: if $H$ is an IHOE of $\Bbbk$, is $H$ necessarily an IHOE of each of its Hopf subalgebras? A closely related question in \cite[Question C]{LZ23} asks whether every finite-GK-dimensional connected Hopf algebra is an IHOE of its primitive part, the universal enveloping algebra of its primitive space. The umbrella Hopf algebra $\mathrm{UM}(2,2)$, constructed in \cite{HW26} and revisited in Remark \ref{example-order-preserving}, serves as a counterexample to both questions: it is an IHOE of $\Bbbk$ but not of primitive part \cite[Corollary 4.5]{HW26}.
	
	In addition, Hu proved that finite-GK-dimensional connected Hopf algebras are closed under Hopf algebra extensions \cite{Hu26}. Together with the subquotient closure of IHOEs established above, this motivates the following  question concerning extension stability of the IHOE framework.
	
	\begin{question}
		For any exact sequence of Hopf algebras $\Bbbk \rightarrow K \rightarrow H \rightarrow J \rightarrow \Bbbk,$
		if $K$ and $J$ are IHOEs of $\Bbbk$, is $H$ also an IHOE of $\Bbbk$?
	\end{question}

	Beyond Hopf subalgebras and quotient Hopf algebras, one-sided coideal subalgebras represent another fundamental class of intermediate substructures that merit systematic study. A \emph{left} (resp.\ \emph{right}) coideal subalgebra of a Hopf algebra $H$ is a subalgebra $A\subseteq H$ satisfying $\Delta_H(A)\subseteq H\otimes A$ (resp.\ $\Delta_H(A)\subseteq A\otimes H$). Geometrically, such subalgebras correspond to quantum homogeneous spaces, and existing literature has mainly focused on their freeness, faithful flatness, and homological properties; see \cite{Ma91,Sk07,LW12,Kr12}. From a Lie-theoretic perspective, they serve as the core constituents of quantum symmetric pairs, a notion originally introduced by Letzter for quantized enveloping algebras \cite{Le02} and further advanced by Wang and his collaborators \cite{Wa23}. In contrast to Hopf subalgebras, which are generally scarce for quantum groups, coideal subalgebras provide a rich collection of intermediate substructures \cite{Kh08,Kh11,KS08,Po09,HS13}. For finite-GK-dimensional connected Hopf algebras, additional ring-theoretic and homological properties of coideal subalgebras have been obtained in \cite{BG16,Zh24}. In the present work, we establish a structure characterization of coideal subalgebras within arbitrary IHOEs of $\Bbbk$ from the perspective of Ore extension theory, which forms our second main theorem.

	\begin{thm}[Theorem \ref{Hopf-subalgebra-structure}]
		\label{main-theorem-2}
		Let $H$ be an IHOE of $\Bbbk$. Then every left (resp.\ right) coideal subalgebra $A$ of $H$ is an iterated Ore extension of $\Bbbk$. More precisely, there exists a chain
		\[
		\Bbbk =A_{0} \subset A_{1} \subset \cdots \subset A_{m}=A
		\]
		of left (resp.\ right) coideal subalgebras of $H$ with $m=\operatorname{GKdim} A$, such that each $A_i$ is an Ore extension of $A_{i-1}$ for $1\le i\le m$.
	\end{thm}

	The proofs of the above structural theorems rely fundamentally on Kharchenko’s \emph{thin replacement machinery} for PBW generating systems, a perturbation technique originally developed for coideal subalgebras of quantum Borel algebras \cite{Kh08,Kh11,KS08,Po09}. To set the stage, we fix necessary notation and terminology. Let $H$ be an algebra equipped with an algebra comultiplication $\Delta:H\to H\otimes H$. Let $\GG = (\{z_{\xi}\}_{\xi\in \Xi}, <)$ be an ordered family of elements in $H$, i.e., $\{z_{\xi}\}_{\xi\in \Xi}$ is a family of elements of $H$ and $<$ is a total order on the index set $\Xi$. The family $\GG$ forms a \emph{PBW generating system} if all ordered monomials in $\{z_\xi\}$ constitute a linear basis of $H$ (Definition \ref{definition-PBW-generating-system}). Let $\Gamma$ be a nontrivial well-ordered free abelian monoid and $\dd:\Xi\to\Gamma$ a positive weight function. The family $\GG$ is \emph{$\dd$-triangular} and \emph{$(\Delta,\dd)$-dominated} whenever the multiplication and comultiplication of generators admit specific constrained expansions in ordered monomials, respectively (Definition \ref{definition-triangular}). A \emph{$\dd$-thin replacement} of $\GG$ is a perturbed ordered family $\hat{\GG}=(\{\hat{z}_\xi\}_{\xi\in\Xi},<)$, where each generator $\hat{z}_\xi$ differs from $z_\xi$ by a linear combination of  constrained ordered monomials (Definition \ref{thin replacement}). For any subalgebra $A\subseteq H$, we define $\Xi(A,\GG,\dd)$  as the set of indices $\xi\in \Xi$ such that $z_{\xi}$ admits a thin perturbation contained in $A$ (Formula \eqref{sub-index-set}). Based on these notions, we establish our third core theorem, which governs the inheritance of PBW structures under coideal subalgebra restriction.

	\begin{thm}[PBW Inheritance Theorem, Theorem \ref{PBW-sub-hereditary}]\label{main-theorem-PBW-inheritance}
		Let $H$ be an algebra equipped with a PBW generating system $\mathbb{G}=(\{z_{\xi}\}_{\xi\in \Xi},<)$.
		Let $A\subseteq H$ be a subalgebra such that either $\Delta(A)\subseteq A\otimes H$ or $\Delta(A)\subseteq H\otimes A$ with respect to some algebra homomorphism $\Delta:H\to H\otimes H$.
		Assume $\mathbb{G}$ is $\dd$-triangular and $(\Delta,\dd)$-dominated with respect to some order-preserving positive map $\dd\colon \Xi\to \Gamma$.
		If $\hat{\mathbb{G}}=(\{\hat{z}_{\xi}\}_{\xi\in \Xi},<)$ is any $\dd$-thin replacement of $\mathbb{G}$ satisfying $\hat{z}_{\xi}\in A$ for all $\xi\in \Xi(A,\mathbb{G},\dd)$,
		then the restricted ordered family
		\[
		\hat{\mathbb{G}}\big|_{\Xi(A,\mathbb{G},\dd)}
		=\bigl(\{\hat{z}_{\xi}\}_{\xi\in \Xi(A,\mathbb{G},\dd)},\,{<}\big|_{\Xi(A,\mathbb{G},\dd)}\bigr)
		\]
		constitutes a PBW generating system for $A$ and is $\dd\big|_{\Xi(A,\mathbb{G},\dd)}$-triangular.
	\end{thm}
	
	With this PBW inheritance theorem in hand, we now outline the proof strategy unifying the Hopf subalgebra case of Theorem \ref{main-Sub-Quotient} and the coideal subalgebra characterization from Theorem \ref{main-theorem-2}. Separately, we address the quotient Hopf algebra case of Theorem \ref{main-Sub-Quotient}. Let $H$ be an IHOE of $\Bbbk$. By construction, $H$ possesses a canonical finite PBW generating system that is $\dd$-triangular and $(\Delta,\dd)$-dominated for some order-preserving positive map $\dd\colon \Xi\to \mathbb{N}$. For any Hopf subalgebra or one-sided coideal subalgebra $A\subseteq H$, Theorem \ref{main-theorem-PBW-inheritance} endows $A$ with a $\dd|_{\Xi'}$-triangular PBW generating system $\hat{\GG}|_{\Xi'}$, where $\Xi'=\Xi(A,\GG,\dd)$. Applying Proposition \ref{PBW-imply-IOE} to this restricted PBW generating system produces an Ore extension tower for $A$ of length precisely $\operatorname{GKdim} A$. Furthermore, by Proposition \ref{proposition of replacement}, the $\dd$-thin replacement $\hat{\GG}$ inherits both $\dd$-triangular and $(\Delta,\dd)$-dominated properties, ensuring that every layer of the resulting Ore extension tower preserves the structural type of $A$. 
	
	We next resolve the quotient case of Theorem \ref{main-Sub-Quotient}. By Masuoka’s correspondence theorem \cite{Ma91}, any Hopf ideal $\mathfrak{a} \subset H$ corresponds to a left coideal subalgebra $A:=H^{\operatorname{co} H/\mathfrak{a}}$ of $H$ such that $\mathfrak{a}=A^{+}H$. Applying Theorem \ref{main-theorem-PBW-inheritance}, we obtain an ordered family $\overline{\GG}=(\{\overline{z}_{\xi}\}_{\xi\in \Xi''},\,{<}|_{\Xi''})$ in $H/\mathfrak{a}$, where $\Xi''=\Xi\setminus\Xi'$ and $\overline{z}_{\xi}=\hat z_{\xi}+\mathfrak{a}$. This family forms a $\dd|_{\Xi''}$-triangular and $(\Delta_{H/\mathfrak{a}},\dd|_{\Xi''})$-dominated PBW generating system for the quotient Hopf algebra $H/\mathfrak{a}$. Consequently, Proposition \ref{PBW-imply-IHOE} implies that $H/\mathfrak{a}$ is an IHOE of $\Bbbk$, which completes the proof of the quotient closure property.

	Building on the preceding structural and technical theorems, we develop a systematic combinatorial procedure to classify coideal subalgebras in arbitrary IHOEs, which we formulate as the algorithmic theorem below.
	As concrete applications of this algorithm, we completely classify all right coideal subalgebras of the connected Hopf algebras $A(\lambda_{1},\lambda_{2},\alpha)$ and $B(\lambda)$, representative families of noncocommutative connected Hopf algebras of GK-dimension three originally classified by Zhuang~\cite{Zh13}; see Subsections \ref{subsec:A-type} and \ref{subsec:B-type}.
	
	\begin{thm}[Classification Algorithm for Coideal Subalgebras, Subsection \ref{subsec:algorithm}]
		\label{thm-algorithm}
		There exists an explicit finite combinatorial algorithm that takes as input any IHOE of $\Bbbk$ and outputs a complete list of all left and right coideal subalgebras of the given IHOE.
	\end{thm}
	
	The remainder of this paper is organized as follows. Section \ref{sec:prelim} reviews preliminary notation, PBW generating systems, and Ore extension theory. Section \ref{sec:triangular} develops the triangular and dominated axioms for PBW generating systems. Section \ref{sec:thin-replacement} establishes the thin replacement machinery and proves the PBW inheritance theorem (Theorem \ref{main-theorem-PBW-inheritance}). Section \ref{sec:IHOE-structure} investigates iterated Hopf Ore extensions, establishes subquotient closure (Theorem \ref{main-Sub-Quotient}) and the coideal subalgebra structure theorem (Theorem \ref{main-theorem-2}). Complementing these structural results, we introduce minimal PBW generating systems to yield an explicit description of coradical filtrations, together with a sharp criterion determining when an IHOE of $\Bbbk$ is an IHOE  of  its primitive part. Section \ref{sec:classification} presents our  classification algorithm and furnishes a complete classification of  right coideal subalgebras for noncocommutative  connected Hopf algebras of GK-dimension three.

	\paragraph{Notation and conventions.}
	Throughout this paper, we work over a fixed base field $\Bbbk$, which is tacitly assumed to be of characteristic zero. All vector spaces, algebras, Hopf algebras, and unadorned tensor products are defined over $\Bbbk$. We let $\mathbb{N}$ denote the set of natural numbers. For each integer $n\ge 0$, write $\II_{n} = \{1,\dots, n\}$. The cardinality of a set $X$ is denoted by $\#(X)$.

	\section{Preliminaries}
	\label{sec:prelim}
	
	In this section, we fix basic definitions, conventions and notation employed throughout the paper. In particular, we recall the notions of PBW generating systems and Ore extensions for algebras. Furthermore, since words over totally ordered sets are used to index algebraic elements, we introduce several order relations on words that are frequently employed in the sequel.

	\subsection{PBW generators and Ore extensions}
	Let $H$ be an algebra. Let $\GG = (\{z_{\xi}\}_{\xi\in \Xi}, <)$ be an \emph{ordered family} of elements in $H$. This means that $\{z_{\xi}\}_{\xi\in \Xi}$ is a family of elements of $H$ and $<$ is a total order on the index set $\Xi$.
	
	We write $\PP(\Xi)$ for the set of words on $\Xi$. A \emph{word} on $\Xi$ is a finite sequence of elements of $\Xi$. The empty word is denoted by $\emptyset$. Set $\PP_{+}(\Xi) = \PP(\Xi)\setminus\{\emptyset\}$. Define 
	\[
	z_{V} := z_{\xi_{1}}\cdots z_{\xi_{n}}
	\]
	for any nonempty word $V=\xi_{1}\cdots \xi_{n}\in \PP_{+}(\Xi)$.
	By convention, $z_{\emptyset} = 1$.
	
	Let $\PP(\Xi, <)$ denote the set of non-decreasing words on $\Xi$ with respect to $<$. A word $\xi_{1}\cdots\xi_{n}$ lies in $\PP(\Xi,<)$ if and only if $\xi_{1}\leq \cdots\leq \xi_{n}$. Write $\PP_{+}(\Xi,<) = \PP(\Xi,<)\setminus\{\emptyset\}$. Let
	\begin{eqnarray}\label{rearrangement-map}
		\Pi=\Pi_{<}\colon \PP(\Xi)\to \PP(\Xi,<)
	\end{eqnarray}
	be the rearrangement map which sends each word $V=\xi_{1}\cdots\xi_{n}$ to the unique permutation $\Pi(V)=\xi_{i_{1}}\cdots\xi_{i_{n}}$ satisfying $\xi_{i_{1}}\leq \cdots\leq \xi_{i_{n}}$. We set $\Pi(\emptyset)=\emptyset$.
	
	\begin{definition}\label{definition-PBW-generating-system}
		A \emph{PBW generating system} (or \emph{system of PBW generators}) for $H$ is an ordered family
		$\GG=(\{z_{\xi}\}_{\xi\in \Xi}, <)$ of elements of $H$ such that the family $\{z_V\}_{V\in \PP(\Xi,<)}$ forms a basis of $H$. The system is called \emph{finite} if the index set $\Xi$ is finite.
		
		Now let $H=\bigoplus_{\gamma\in \Gamma}H_{\gamma}$ be a $\Gamma$-graded algebra, where $\Gamma$ is an abelian monoid. A PBW generating system $\GG=(\{z_{\xi}\}_{\xi\in \Xi}, <)$ of $H$ is \emph{homogeneous} whenever each $z_{\xi}$ is homogeneous. In this case, the family $\{z_V\}_{V\in \PP(\Xi,<)}$ constitutes a homogeneous basis of $H$.
	\end{definition}
	
	Suppose $\GG=(\{z_\xi\}_{\xi \in\Xi},<)$ is a system of PBW generators of $H$. Then every $f\in H$ and every $g\in H\otimes H$ admit unique expressions
	\[
	f= \sum_{V\in \PP(\Xi,<)} a_{V} z_{V} \quad \text{and} \quad g = \sum_{(P,Q)\in \PP(\Xi, <)^{2}} b_{P,Q} z_{P} \otimes z_{Q}
	\]
	with coefficients $a_{V}, b_{P,Q}\in \Bbbk$. The sets
	\[
	\supp_{\GG}(f)=\{V \in \PP(\Xi,<) \mid a_{V}\neq0\} \quad \text{and} \quad \supp_{\GG}(g) = \{ (P,Q) \in \PP(\Xi, <)^{2} \mid b_{P,Q} \neq 0 \}
	\]
	are called the \emph{$\GG$-support} of $f$ and $g$, respectively. For any $U\in \PP(\Xi)$, let
	\[
	\phi_{U}\colon H\to \Bbbk
	\]
	denote the canonical linear map defined by $z_{V} \mapsto \delta_{U,V}$, where $\delta_{U,V}$ denotes the Kronecker delta.

	\begin{lemma}\label{GKdimension}
		Let $H=\bigoplus_{\gamma\in \Gamma}H_{\gamma}$ be a connected $\Gamma$-graded algebra with $\Gamma$ a nontrivial free abelian monoid. Let $\GG=(\{z_{\xi}\}_{\xi\in \Xi},<)$ be a homogeneous PBW generating system of $H$. Then
		\[
		\GKdim H = \#(\Xi),
		\]
		where $\#(\Xi)$ stands for the cardinality of $\Xi$.
	\end{lemma}
	\begin{proof}
		By \cite[Lemma 1.4]{Zh24}, we have $\GKdim H\geq \#(\Xi)$. Hence it suffices to treat the case $\#(\Xi)<\infty$. Let $\varphi\colon \Gamma\to \mathbb{N}$ be any monoid homomorphism satisfying $\varphi^{-1}(0)=\{0\}$. Such a homomorphism exists since $\Gamma$ is a free abelian monoid. Define the $\mathbb{N}$-graded algebra $\varphi_{!}(H)$ which coincides with $H$ as an algebra, with grading
		\[
		\varphi_{!}(H)_{m} = \sum_{\gamma\in \varphi^{-1}(m)} H_{\gamma},\quad m\geq0.
		\]
		Clearly $\varphi_{!}(H)$ is a connected $\mathbb{N}$-graded algebra with Hilbert series
		\[
		h_{\varphi_{!}(H)}(t) = \prod_{\xi\in \Xi} \big (1-t^{\deg(z_{\xi})}\big)^{-1}.
		\]
		We deduce
		\[
		\GKdim H=\GKdim \varphi_{!}(H) = \varlimsup_{n\to \infty} \log_n \dim \varphi_{!}(H)_{\leq n} = \#(\Xi).
		\]
		The last two equalities follow respectively from \cite[Lemma 6.1 (b)]{KL00} and \cite[Proposition 2.21]{ATV91}.
	\end{proof}

	\begin{definition}\label{Ore extension}
		Let $R$ be an algebra.
		\begin{enumerate}
			\item An \emph{Ore extension} of $R$ is an algebra $H$ containing $R$ as a subalgebra for which there exist $X\in H$, an automorphism $\sigma$ of $R$, and a $\sigma$-derivation $\delta$ of $R$ satisfying:
			\begin{enumerate}
				\item $H$ is a free left $R$-module with basis $\{1,X,X^{2},\dots\}$;
				\item $Xr=\sigma(r)X+\delta(r)$ for all $r\in R$.
			\end{enumerate}
			We write $R[X;\sigma,\delta]$ for such an Ore extension. We adopt the shorthand $H=R[X;\sigma]$ (resp.\ $H=R[X;\delta]$) when $\delta=0$ (resp.\ $\sigma=\operatorname{id}_R$).
			
			\item An \emph{$n$-step iterated Ore extension} (\emph{$n$-step IOE}) of $R$, where $n\geq 0$ is an integer, is an algebra $H$ equipped with an ascending chain of subalgebras
			\[
			R=H_{0} \subset H_{1} \subset \cdots \subset H_{n} =H
			\]
			such that $H_{i}$ is an Ore extension of $H_{i-1}$ for each $i=1,\dots, n$. The notation
			\[
			H=R[X_{1};\sigma_{1},\delta_{1}]\cdots [X_{n};\sigma_{n},\delta_{n}]
			\]
			indicates that $H_{0}=R$, $H_{n}=H$, and $H_{i}= H_{i-1}[X_{i};\sigma_{i},\delta_{i}]$ for all $i=1,\dots,n$. 
		\end{enumerate}
		Note that the $0$-step IOE of $R$ is $R$ itself, and a $1$-step IOE coincides with an Ore extension of $R$. We say that $H$ is an \emph{IOE}  of $R$ if $H$ is an $n$-step IOE of $R$ for some integer $n\geq 0$.
	\end{definition}
	
	Graded Ore extensions and graded iterated Ore extensions  of graded algebras are defined analogously: one requires $X,\sigma,\delta$ and $X_{i},\sigma_{i},\delta_{i}$ to be homogeneous of suitable degrees.
	
	\begin{lemma}\label{IOE facts}
		Let $H=\Bbbk[X_{1};\sigma_{1},\delta_{1}]\cdots[X_{n};\sigma_{n},\delta_{n}]$ be an IOE of $\Bbbk$. Let $z_{i} \in X_{i} + H_{i-1}$ for $i=1,\ldots, n$. Then the ordered family $\GG=(\{z_{i}\}_{i\in \II_{n}}, <)$ is a system of PBW generators of $H$, where $\II_{n}=\{1,\ldots, n\}$ and $<$ is the natural total order on $\II_{n}$ given by $1<2<\cdots< n$.
	\end{lemma}
	
	\begin{proof}
		The assertion follows directly.
	\end{proof}

	\subsection{Order relations on words}

	Recall that a partial order $<$ on a set $X$ is \emph{well-founded} if every nonempty subset of $X$ possesses a minimal element for $<$, or equivalently, if there exists no infinite strictly descending chain $x_{1}> x_{2}> x_{3}> \cdots$ in $X$. A \emph{well order} is defined as a well-founded total order.
	
	Let $(\Xi, <)$ be a totally ordered set. Denote by $\PP(\Xi)$ the set of all words on $\Xi$. This set forms a monoid under word concatenation. The concatenation of two words $U$ and $V$ is denoted by $UV$. We say that $U$ is a \emph{prefix} (resp.\ \emph{suffix}) of $V$ if $UW = V$ (resp.\ $WU = V$) for some word $W$. A prefix or suffix $U$ of $V$ is said to be \emph{proper} if $U \neq V$.

	We first extend $<$ to partial orders $<_{\mathfrak{l}}$ and $<_{\rr}$ on $\PP(\Xi)$ as follows. For $U,V\in \PP(\Xi)$,
	\begin{eqnarray*}
		U<_{\mathfrak{l}} V &\Longleftrightarrow&
		U=W_{1}\eta W_{2},\quad V=W_{1}\xi W_{3},\quad \eta, \xi \in\Xi,\ \eta<\xi,\ W_{i}\in\PP(\Xi);\\
		U<_{\rr } V &\Longleftrightarrow&
		U=W_{1}\eta W_{3},\quad V=W_{2}\xi W_{3},\quad \eta, \xi \in\Xi,\ \eta<\xi,\ W_{i}\in\PP(\Xi).
	\end{eqnarray*}
	For all words $U,V,W_{1},W_{2},W_{3}$, we have
	\begin{eqnarray*}
		U<_{\mathfrak{l}} V    \Longrightarrow  W_{1} U W_{2} <_{\mathfrak{l}} W_{1} V W_{3}, \qquad
		U<_{\rr } V    \Longrightarrow  W_{1} U W_{3} <_{\rr } W_{2} V W_{3}.
	\end{eqnarray*}
	Neither $<_{\mathfrak{l}}$ nor $<_{\rr}$ is a total order.
	
	The \emph{left lexicographic order} $<_{\mathfrak{l}, \lex}$ and the
	\emph{right lexicographic order} $<_{\rr , \lex}$ on $\PP(\Xi)$
	relative to $<$ are refinements of $<_{\mathfrak{l}}$ and $<_{\rr}$, respectively. They are defined by: for $U,V\in \PP(\Xi)$,
	\begin{eqnarray*}
		U<_{\mathfrak{l},\lex} V & \Longleftrightarrow &
		U\text{ is a proper prefix of }V \text{ or } U <_{\mathfrak{l}} V; \\
		U<_{\rr ,\lex} V & \Longleftrightarrow &
		U\text{ is a proper suffix of }V \text{ or } U <_{\rr } V.
	\end{eqnarray*}
	The order $<_{\mathfrak{l}, \lex}$ (resp.\ $<_{\rr , \lex}$) is compatible with concatenation on the left (resp.\ right), but not on the opposite side. Both are total orders; however, neither is a well order whenever $\Xi$ contains more than one element.
	
	\begin{lemma}\label{order-compare-rearrangement}
		Let $U_{1}, U_{2}, V_{1}, V_{2} \in \PP(\Xi)$.
		\begin{enumerate}
			\item If $\Pi(U_{1})\leq _{\mathfrak{l},\lex} \Pi(V_{1})$ and $\Pi(U_{2})<_{\mathfrak{l}} \Pi(V_{2})$, then $\Pi(U_{1}U_{2}) <_{\mathfrak{l}} \Pi(V_{1} V_{2})$;
			\item If $\Pi(U_{1})\leq _{\mathfrak{l},\lex} \Pi(V_{1})$ and $\Pi(U_{2})<_{\mathfrak{l},\lex} \Pi(V_{2})$, then $\Pi(U_{1}U_{2}) <_{\mathfrak{l},\lex} \Pi(V_{1} V_{2})$;
			\item If $\Pi(U_{1})\leq _{\rr ,\lex} \Pi(V_{1})$ and $\Pi(U_{2})<_{\rr } \Pi(V_{2})$, then $\Pi(U_{1}U_{2}) <_{\rr } \Pi(V_{1} V_{2})$;
			\item If $\Pi(U_{1})\leq _{\rr ,\lex} \Pi(V_{1})$ and $\Pi(U_{2})<_{\rr ,\lex} \Pi(V_{2})$, then $\Pi(U_{1}U_{2}) <_{\rr ,\lex} \Pi(V_{1} V_{2})$.
		\end{enumerate}
		Here, $\Pi:\PP(\Xi)\to \PP(\Xi,<)$ is  the rearrangement map given in \eqref{rearrangement-map}.
	\end{lemma}
	\begin{proof}
		Verification is straightforward.
	\end{proof}
	
	In the sequel, let $\Gamma= (\Gamma,<)$ denote a nontrivial well-ordered free abelian monoid. The well order $<$ satisfies the monotonicity condition: for all $\gamma_{1}, \gamma_{2}, \gamma \in \Gamma$,
	\[
	\gamma_{1} <\gamma_{2} \Longrightarrow \gamma_{1} +\gamma <\gamma_{2} +\gamma.
	\]
	Let $\dd\colon\PP(\Xi)\to \Gamma$ be a monoid homomorphism. Define total orders $<_{\mathfrak{l}, \dd}, \, <_{\rr , \dd}$ on $\PP(\Xi)$ as follows. For $U,V\in \PP(\Xi)$,
	\begin{eqnarray*}
		U <_{\mathfrak{l}, \dd} V \quad(\text{resp. } U <_{\rr , \dd} V)
		& \Longleftrightarrow &
		\left\{
		\begin{aligned}
			& \dd(U)<\dd(V), \quad \text{or} \\
			&\dd(U)=\dd(V) \text{ and } U <_{\mathfrak{l}} V \quad(\text{resp. } U <_{\rr } V).
		\end{aligned}\right .
	\end{eqnarray*}
	These orders are compatible with word concatenation on both sides.
	
	\begin{lemma}\label{wellorder}
		Let $\dd\colon\PP(\Xi)\to \Gamma$ be a monoid homomorphism satisfying $\dd(\xi)\neq 0$ for all $\xi\in \Xi$. Suppose that for each $\gamma\in \Gamma$, the restriction of $<$ to the subset $\{\xi \in \Xi \mid\dd(\xi)=\gamma\}$ is well-founded. Then the total orders $<_{\mathfrak{l}, \dd}$ and $<_{\rr , \dd}$ on $\PP(\Xi)$ are well-founded.
	\end{lemma}
	
	\begin{proof}
		The argument adapts the proof of \cite[Lemma 3.1]{LZ25}. We treat $<_{\mathfrak{l}, \dd}$; the case of $<_{\rr , \dd}$ is analogous. Since $(\Gamma,<)$ is well-ordered, it suffices to show that any non-increasing sequence
		\[
		V_1 \geq_{\mathfrak{l}, \dd} V_2 \geq_{\mathfrak{l}, \dd} V_3 \geq_{\mathfrak{l}, \dd} \cdots
		\]
		with constant $\dd(V_i)=\gamma$ eventually stabilizes. We proceed by induction on $\gamma\in\Gamma$ with respect to $<$. The case $\gamma=0$ is trivial, as all $V_i=\emptyset$. Now assume $\gamma>0$, so each $V_{i}\neq \emptyset$. Write $V_{i}=\xi_{i}W_{i}$, where $\xi_{i}\in\Xi$ is the first letter of $V_i$ and $W_{i}\in\PP(\Xi)$. We obtain
		\[
		\xi_{1}\geq\xi_{2} \geq\xi_{3} \geq\cdots.
		\]
		Set $\Gamma_{\gamma}:=\{\alpha\in \Gamma \mid \exists\,\beta\in\Gamma,\ \alpha+\beta =\gamma\}$. The set $\Gamma_{\gamma}$ is finite, and each $\dd(\xi_{i})\in \Gamma_{\gamma}$. By the hypothesis on $<$, there exists $p\ge1$ such that $\xi_{p}=\xi_{p+1}=\cdots$. Consequently
		\[
		W_p\geq_{\mathfrak{l}, \dd} W_{p+1}\geq_{\mathfrak{l}, \dd} \cdots,
		\]
		and $\dd(W_{p})=\dd(W_{p+1})=\cdots < \gamma$. The induction hypothesis yields stabilization.
	\end{proof}
	
	\section{Structural conditions for PBW generators}
	\label{sec:triangular}
	
	In this section, we introduce fundamental structural properties of PBW generating systems and establish several key observations, which serve as cornerstones of this paper. As will be seen in Section \ref{sec:IHOE-structure}, these properties hold for the associated PBW generating systems of iterated Hopf Ore extensions of $\Bbbk$. Consequently, the results and techniques developed in this and the subsequent section can be applied to investigate the structure of iterated Hopf Ore extensions of $\Bbbk$.

	Throughout this section, unless specified otherwise, $H$ denotes an algebra, $\GG=(\{z_{\xi}\}_{\xi\in \Xi}, <)$ stands for an ordered family of elements in $H$, $\Delta\colon H\to H\otimes H$ is an algebra homomorphism, $\Gamma=(\Gamma,<)$ denotes a nontrivial well-ordered free abelian monoid, and $\dd\colon \Xi \to \Gamma$ is a \emph{positive map}, meaning $\dd(\xi)>0$ for all $\xi\in \Xi$. Notation and conventions from the previous section remain in force. In particular, recall that $z_{V}:=z_{\xi_{1}}\cdots z_{\xi_{n}}$ for any word $V=\xi_{1}\cdots \xi_{n}\in \PP(\Xi)$. Moreover, $\Pi\colon \PP(\Xi) \to \PP(\Xi,<)$ denotes the rearrangement map.
	
	\begin{convention}
		Recall that $\PP(\Xi)$ is the free monoid on $\Xi$ with multiplication given by word concatenation. Therefore, for any map $\dd\colon \Xi \to \Gamma$, there exists a unique monoid homomorphism $\overline{\dd}\colon \PP(\Xi) \to \Gamma$ satisfying $\overline{\dd}|_{\Xi}=\dd$. Explicitly,
		$\overline{\dd}(\xi_{1}\cdots\xi_{n})=\dd(\xi_{1})+\cdots+\dd(\xi_{n})$
		for every word $\xi_{1}\cdots \xi_{n}\in \PP(\Xi)$. In what follows, we simply write $\dd$ instead of $\overline{\dd}$ by abuse of notation.
	\end{convention}
	
	\subsection{Definitions and basic examples}
	
	\begin{definition}\label{definition-triangular}
		An ordered family $\GG=(\{z_{\xi}\}_{\xi\in \Xi}, <)$ of elements of $H$ is
		\begin{enumerate}
			\item called \emph{triangular} if for any pair $\eta,\xi \in \Xi$ with $\eta < \xi$,
			\[
			z_\xi z_\eta \in \Bbbk^{\times} \cdot z_\eta z_\xi + \Span\big\{z_{W} \,\big|\, W \in \PP_{+}(\Xi, <),\; W<_{\rr,\lex} \eta\xi \big\};
			\]
			\item called \emph{$\dd$-triangular} if for every $\gamma\in \Gamma$, the restriction of $<$ to $\dd^{-1}(\gamma) \subseteq \Xi$ is well-founded, and for any pair $\eta,\xi \in \Xi$ with $\eta < \xi$,
			\[
			z_\xi z_\eta \in  \Bbbk^{\times} \cdot z_\eta z_\xi + \Span\big\{z_{W} \,\big|\, W \in \PP_{+}(\Xi, <),\; \dd(W) \leq \dd(\eta\xi),\; W<_{\rr,\lex} \eta\xi \big\};
			\]
			\item called \emph{strongly $\dd$-triangular} if for every $\gamma\in \Gamma$, the restriction of $<$ to $\dd^{-1}(\gamma) \subseteq \Xi$ is well-founded, and for any pair $\eta,\xi \in \Xi$ with $\eta < \xi$,
			\[
			z_\xi z_\eta \in  \Bbbk^{\times} \cdot z_\eta z_\xi + \Span\big\{z_{W} \,\big|\, W \in \PP_{+}(\Xi, <),\; \dd(W) < \dd(\eta\xi),\; W<_{\rr,\lex} \eta\xi \big\};
			\]
			\item called \emph{$\Delta$-dominated} if for any $\xi\in \Xi$,
			\begin{eqnarray*}
				\Delta(z_\xi) \in 1 \otimes z_\xi+z_\xi \otimes 1  + \Span \big \{z_{P}\otimes z_{Q} \,\big|\, P,Q\in \PP_{+}(\Xi, <),\; P,Q <_{\rr} \xi \big\};
			\end{eqnarray*}
			\item called \emph{$(\Delta,\dd)$-dominated} if for any $\xi\in \Xi$,
			\begin{eqnarray*}
				\Delta(z_\xi) \in 1 \otimes z_\xi+z_\xi \otimes 1  + \Span \big \{z_{P}\otimes z_{Q} \,\big|\, P,Q\in \PP_{+}(\Xi, <),\; \dd(PQ) \leq \dd(\xi),\; P, Q <_{\rr} \xi \big\}.
			\end{eqnarray*}
		\end{enumerate}
	\end{definition}
	
	\begin{example}
		Consider a Cartan matrix $A=(a_{ij})_{i,j\in \II_{n}}$ of size $n$ over $\Bbbk=\mathbb{C}$. Let $(d_{i})_{i\in \II_{n}}$ be the sequence of coprime positive integers such that the matrix $(d_{i}a_{ij})$ is symmetric. The quantized universal enveloping algebra $U_{q}$ attached to $A$ is defined as the algebra generated by $E_{i}, F_{i}, K_{i}, K_{i}^{-1}$ for $1\leq i\leq n$, subject to a set of defining relations. Let $U_{q}^{+}$ denote the positive part of $U_{q}$, i.e., the subalgebra of $U_{q}$ generated by $E_{1},\dots,E_{n}$. A well-behaved system of PBW generators for $U_{q}^{+}$ may be constructed as follows. Let $\Phi\subseteq \mathbb{R}^{n}$ be the root system associated with $A$, and let $W$ be the Weyl group of $\Phi$. The inner product on $\mathbb{R}^{n}$ is specified by
		\[\langle \alpha_{i}, \alpha_{j}\rangle = d_{i} a_{ij},\quad i,j=1,\dots, n,\]
		where $\alpha_{1},\dots,\alpha_{n}$ form the standard basis of $\mathbb{R}^{n}$. Let $w_{0}$ be the longest element of $W$, and take a reduced expression $w_{0}=s_{i_{1}}\cdots s_{i_{N}}$, where $s_{i}\in W$ denotes the simple reflection corresponding to $\alpha_{i}$. Every positive root appears exactly once in the list
		\[\beta_{1}=\alpha_{i_{1}},\quad \beta_{2} = s_{i_{1}}(\alpha_{i_{2}}),\quad \ldots,\quad \beta_{N} = s_{i_{1}}\cdots s_{i_{N-1}}(\alpha_{i_{N}}).\]
		Let $T_{i}\colon U_{q}\to U_{q}$ be Lusztig’s automorphism for each $i=1,\dots, n$. Define
		\begin{eqnarray*}
			z_{1} =E_{i_{1}},\quad z_{2} = T_{i_{1}}(E_{i_{2}}),\quad \ldots,\quad z_{N}= T_{i_{1}}\cdots T_{i_{N-1}}(E_{i_{N}}).
		\end{eqnarray*}
		These elements are known as Lusztig’s root vectors. It is well-known that \[\{z_{1}^{r_{1}}\cdots z_{N}^{r_{N}} \mid r_{1},\dots, r_{N}\geq 0\}\] forms a basis of $U_{q}^{+}$. Furthermore, for all $1\leq i<j\leq N$,
		\[z_{j}z_{i} - q^{-\langle \beta_{i}, \beta_{j} \rangle} z_{i}z_{j} \in \Span\biggl\{ \prod_{\ell=i+1}^{j-1} z_{\ell}^{r_{\ell}} \;\bigg|\; r_{\ell}\geq 0,\;\sum_{\ell=i+1}^{j-1} r_{\ell} \beta_{\ell} =\beta_{i}+\beta_{j}\biggr\}.\]
		This formula is due to Levendorskii–Soibelman \cite{LS91}. Thus, the ordered family $\GG:= \bigl(\{z_{i}\}_{i\in \II_{N}},<\bigr)$, where $<$ denotes the total order on $\II_{N}$ defined by $1< \cdots < N$, forms a triangular system of PBW generators for $U_{q}^{+}$. We refer the reader to \cite[Section 9]{DeP93} for further details.
	\end{example}
	
	In \cite{Kh99}, Kharchenko introduced a combinatorial method built upon Lyndon words, enabling the construction of well-behaved PBW generating systems for algebras subject to mild conditions. This approach was further extended in \cite{ZSL20, Zh24, LZ25}. The following result is an outcome of this method. 
	
	\begin{prop}\label{Graded-imply-PBW}
		Let $H=\bigoplus_{\gamma\in\Gamma} H_{\gamma}$ be a connected $\Gamma$-graded algebra. Suppose $\Delta\colon H\to H\otimes H$ is a homomorphism of $\Gamma$-graded algebras such that
		\[
		\Delta(f)\in 1\otimes f+f\otimes 1+ H_{+} \otimes H_{+}
		\]
		for every homogeneous element $f\in H$, where $H_{+}= \bigoplus_{\gamma\in \Gamma,\gamma\neq 0} H_{\gamma}$. Then $H$ possesses a homogeneous PBW generating system $\GG=(\{z_\xi\}_{\xi \in\Xi},<)$ which is $\dd$-triangular and $(\Delta,\dd)$-dominated, where $\dd\colon \Xi\to \Gamma$ is the positive map given by $\dd(\xi) =\deg(z_{\xi})$.
	\end{prop}
	
	\begin{proof}
		This is a restatement of \cite[Proposition 5.5]{LZ25} under the additional assumption that $\Bbbk$ has characteristic zero; the core statement is established in \cite[Theorem A]{ZSL20}.
	\end{proof}
	
	\begin{lemma}\label{triangular-construct}
		Suppose $\GG=(\{z_\xi\}_{\xi \in\Xi},<)$ is triangular (resp.\ $\Delta$-dominated) and $\Xi$ is finite. Then $\GG$ is strongly $\dd$-triangular (resp.\ $(\Delta,\dd)$-dominated) for any positive map $\dd\colon\Xi\to \Gamma$ satisfying $\dd(\xi)\geq N\cdot \dd(\eta)$ whenever $\eta < \xi$, for some sufficiently large positive integer $N$.
	\end{lemma}
	\begin{proof}
		We first treat the case where $\GG$ is triangular; our argument adapts \cite[Subsection 5.1]{BZ22}. Fix an arbitrary pair $\eta< \xi \in \Xi$. Observe that for a word $W\in \PP(\Xi,<)$, the relation $W<_{\rr,\lex} \eta \xi$ holds if and only if either $W<_{\rr,\lex} \xi$ or $W=W'\xi$ with $W' <_{\rr,\lex} \eta$. Hence there exist finite sets
		\[
		\mathcal{A}(\eta,\xi) \subseteq \{W\in \PP(\Xi, <) \mid W <_{\rr,\lex} \eta\}
		\quad \text{and} \quad
		\mathcal{B}(\eta,\xi) \subseteq \{W\in \PP_{+}(\Xi, <) \mid W <_{\rr,\lex} \xi\}
		\]
		such that
		\[
		z_\xi z_\eta \in \Bbbk^{\times} \cdot z_\eta z_\xi + \Span\{z_{W}z_{\xi} \mid W \in \mathcal{A}(\eta,\xi)\} + \Span\{z_{W} \mid W \in \mathcal{B}(\eta,\xi)\}.
		\]
		Choose a positive integer $N$ with
		\[
		N > \max\big\{ \operatorname{length}(W) \;\big|\; W\in \bigcup_{ \substack{(\eta,\xi) \in \Xi\times \Xi \\ \eta < \xi}} \big( \mathcal{A}(\eta,\xi) \cup \mathcal{B}(\eta,\xi) \big) \big\}.
		\]
		Now let $\dd\colon \Xi\to \Gamma$ be a positive map satisfying the required inequality. For any $W =\eta_{1}\cdots \eta_{s} \in \mathcal{A}(\eta,\xi)$, we have $W\xi <_{\rr,\lex} \eta\xi$ and
		\[
		\dd(W\xi) = \dd(\eta_{1}) +\cdots + \dd(\eta_{s})  +\dd(\xi) < N\cdot  \dd(\eta_{s})  +\dd(\xi) \leq \dd(\eta)  +\dd(\xi)=\dd(\eta\xi).
		\]
		For any $W=\xi_{1}\cdots \xi_{t} \in \mathcal{B}(\eta,\xi)$, we have $W <_{\rr,\lex} \eta\xi$ and
		\[
		\dd(W) = \dd(\xi_{1}) +\cdots +\dd(\xi_{t}) < N\cdot \dd(\xi_{t}) \leq \dd(\xi) < \dd(\eta\xi).
		\]
		Therefore, $\GG$ is strongly $\dd$-triangular, as claimed.
		
		We now turn to the case where $\GG$ is $\Delta$-dominated. By hypothesis, for each $\xi\in \Xi$ there exist finite subsets $\mathcal{C}_{l}(\xi), \mathcal{C}_{r}(\xi) \subseteq \{W\in \PP_{+}(\Xi,<) \mid W <_{\rr} \xi\}$ such that
		\[
		\Delta(z_{\xi}) \in 1\otimes z_{\xi} +z_{\xi} \otimes 1 + \Span\{z_{P}\otimes z_{Q} \mid P\in \mathcal{C}_{l}(\xi),\; Q \in  \mathcal{C}_{r}(\xi)\}.
		\]
		Select a positive integer $N$ satisfying
		\[
		N\geq \max\big\{ \operatorname{length}(PQ) \;\big|\;  (P,Q) \in \bigcup_{\xi\in \Xi} \mathcal{C}_{l}(\xi) \times \mathcal{C}_{r}(\xi) \big\}.
		\]
		Let $\dd\colon \Xi\to \Gamma$ be a positive map satisfying the given condition. For any $P=\xi'_{1}\cdots \xi'_{s} \in \mathcal{C}_{l}(\xi)$ and any $Q=\xi''_{1}\cdots \xi''_{t} \in \mathcal{C}_{r}(\xi)$,
		\[
		\dd(PQ) = \dd(\xi'_{1}) +\cdots + \dd(\xi'_{s}) + \dd(\xi''_{1}) +\cdots +\dd(\xi''_{t}) \leq N\cdot \max\{\dd(\xi'_{s}), \dd(\xi''_{t})\} \leq \dd(\xi).
		\]
		Hence $\GG$ is $(\Delta, \dd)$-dominated.
	\end{proof}
	
	\begin{remark}\label{right-lex-monoid}
		Let $\Xi=\{\xi_{1}, \ldots, \xi_{n}\}$ and equip $\Xi$ with the total order $\xi_{1} < \cdots < \xi_{n}$. Endow $\mathbb{N}^{n}$ with the order $<$ defined as follows: for $\gamma= (\gamma_{1},\ldots, \gamma_{n}), \delta =(\delta_{1},\ldots, \delta_{n})\in \mathbb{N}^{n}$,
		\[
		\gamma<\delta \Longleftrightarrow \gamma\neq \delta \text{ and the right-most nonzero component of $\gamma-\delta$ is negative}.
		\]
		Let $\dd_{\Xi}\colon\Xi\to \mathbb{N}^{n}$ denote the positive map given by $\xi_{i}\mapsto e_{i}$, where $e_{i}\in \mathbb{N}^{n}$ is the standard basis vector with $1$ in coordinate $i$ and zeros elsewhere. For $\GG=(\{z_{\xi}\}_{\xi\in \Xi}, <)$, one readily verifies that triangularity, $\dd_{\Xi}$-triangularity and strong $\dd_{\Xi}$-triangularity are equivalent properties. Similarly, the conditions of being $\Delta$-dominated and $(\Delta,\dd_{\Xi})$-dominated coincide.
	\end{remark}
	
	\subsection{Key structural observations}
	
	We start with the observation that triangular PBW generating systems yield Ore extensions.
	
	\begin{prop}\label{PBW-imply-IOE}
		Let $\GG=(\{z_{\xi}\}_{\xi\in \Xi}, <)$ be a PBW generating system of $H$. For each $\xi\in \Xi$, define
		\[
		H_{\xi}:= \Span\{z_{V}\mid V\in \PP(\Xi_{\xi}, {<}|_{\Xi_{\xi}} ) \} \quad \text{and} \quad H_{\xi}^{o}:= \Span\{z_{V}\mid V\in \PP(\Xi_{\xi}^{o}, {<}|_{\Xi_{\xi}^{o}}) \},
		\]
		where $\Xi_{\xi} = \{\eta\in \Xi \mid \eta\leq  \xi\}$ and $\Xi_{\xi}^{o} = \{\eta\in \Xi \mid \eta< \xi\}$. Suppose $\GG$ is $\dd$-triangular with respect to some positive map $\dd\colon\Xi\to \Gamma$. Then $H_{\xi}$ and $H_{\xi}^{o}$ are subalgebras of $H$, and $$H_{\xi} = H_{\xi}^{o}[z_{\xi}; \sigma_{\xi}, \delta_{\xi}]$$ for some algebra automorphism $\sigma_{\xi}$ and $\sigma_{\xi}$-derivation $\delta_{\xi}$ of $H_{\xi}^{o}$, for every $\xi\in \Xi$.
	\end{prop}
	
	The proof relies on the following technical rearrangement lemma.
	
	\begin{lemma}\label{rearrangement}
		Assume $\GG=(\{z_{\xi}\}_{\xi\in \Xi}, <)$ is $\dd$-triangular. Then for any $V\in \PP(\Xi)$,
		\[
		z_V \in \Bbbk^{\times} \cdot z_{\Pi(V)} +  \Span\big\{z_{W} \,\big|\, W \in \PP_{+}(\Xi, <),\; \dd(W) \leq \dd(V),\; W<_{\rr,\lex} \Pi(V) \big\}.
		\]
	\end{lemma}
	
	\begin{proof}
		We proceed by induction on $V$ with respect to $<_{\mathfrak{l}, \dd}$, which is a well-order on $\PP(\Xi)$ by Lemma \ref{wellorder}. The claim trivially holds for $V=\emptyset$, the minimal element with respect to $<_{\mathfrak{l}, \dd}$. Now suppose $V\neq \emptyset$. If $V\in \PP(\Xi, <)$ then $\Pi(V)=V$ and the statement holds immediately. We may therefore assume $V$ admits a decomposition $V=V_{1}\xi \eta V_{2}$ with $V_{1}, V_{2} \in \PP(\Xi)$ and $\eta,\xi \in \Xi$ satisfying $\eta< \xi$. By the $\dd$-triangularity hypothesis,
		\begin{eqnarray*}
			z_{V}= a \cdot z_{V_{1}} z_{\eta} z_{\xi} z_{V_{2}}  + \sum_{j=1}^{m} b_{j}\cdot z_{V_{1}} z_{M_{j}} z_{V_{2}} = a\cdot z_{V_{1} \eta\xi V_{2}} +\sum_{j=1}^{m} b_{j} \cdot z_{V_{1}M_{j}V_{2}},
		\end{eqnarray*}
		where $m\geq 0$, $a\in \Bbbk^{\times}$, $b_{j}\in \Bbbk$, and each $M_{j} \in \PP(\Xi,<)$ satisfies $\dd(M_{j}) \leq \dd(\eta\xi)$ and $M_{j}<_{\rr,\lex} \eta \xi$. Clearly, $\dd(V_{1}M_{j} V_{2}) \leq \dd(V_{1} \eta \xi V_{2}) =\dd(V)$. By Lemma \ref{order-compare-rearrangement},
		\[\Pi(V_{1}M_{j} V_{2}) <_{\rr,\lex} \Pi(V_{1} \eta \xi V_{2}) = \Pi(V).\]
		Since $V_{1}\eta\xi V_{2}, V_{1} M_{j}V_{2} <_{\mathfrak{l}, \dd} V$, the induction hypothesis yields the desired conclusion.
	\end{proof}
	
	\begin{proof}[Proof of Proposition \ref{PBW-imply-IOE}]
		By Lemma \ref{rearrangement}, $H_{\xi}$ and $H_{\xi}^{o}$ are subalgebras of $H$.
		
		The families $\{z_{V}\}_{V\in \PP(\Xi_{\xi}, {<}|_{\Xi_{\xi}})}$ and $\{z_{V}\}_{V\in \PP(\Xi_{\xi}^{o}, {<}|_{\Xi_{\xi}^{o}})}$ are linearly independent, hence form bases of $H_{\xi}$ and $H_{\xi}^{o}$ respectively. It follows that $\{1, z_{\xi}, z_{\xi}^{2}, \ldots\}$ is a basis of $H_{\xi}$ as a left $H_{\xi}^{o}$-module. Applying Lemma \ref{rearrangement} again, for any $V\in \PP(\Xi_{\xi}^{o}, {<}|_{\Xi_{\xi}^{o}})$ we obtain
		\[z_{\xi}z_{V} =(q_{V}z_{V}+g_{V})z_{\xi} + h_{V},\]
		where $q_{V}\in \Bbbk^{\times}$, $h_{V} \in H_{\xi}^{o}$ and
		\[g_{V} \in \Span\big\{z_{W}\,\big|\, W\in \PP(\Xi_{\xi}^{o}, {<}|_{\Xi_{\xi}^{o}}),\; \dd(W)\leq \dd(V),\; W<_{\rr,\lex} V\big\}.\]
		Define linear maps $\sigma_{\xi}, \delta_{\xi}\colon H_{\xi}^{o} \to H_{\xi}^{o}$ by
		\[
		\sigma_{\xi}(z_{V}) = q_{V}z_{V}+g_{V} \quad \text{and} \quad  \delta_{\xi}(z_{V}) =h_{V}
		\]
		for each $V\in \PP(\Xi_{\xi}^{o}, {<}|_{\Xi_{\xi}^{o}})$. One immediately sees that $\sigma_{\xi}\colon H_{\xi}^{o}\to H_{\xi}^{o}$ is lower-triangular with nonzero diagonal entries with respect to the well-ordered basis $\{z_{V}\}_{V\in \PP(\Xi_{\xi}^{o}, {<}|_{\Xi_{\xi}^{o}})}$, where the index set $\PP(\Xi_{\xi}^{o}, {<}|_{\Xi_{\xi}^{o}})$ carries the restriction of $<_{\rr,\dd}$ from $\PP(\Xi, <)$. Consequently, $\sigma_{\xi}$ is a linear isomorphism. Moreover, $z_{\xi}f = \sigma_{\xi}(f) z_{\xi} + \delta_{\xi}(f)$ for all $f\in H_{\xi}^{o}$, and
		\[
		\sigma_{\xi}(fg) z_{\xi} +\delta_{\xi}(fg) = z_{\xi}(fg) = \big(\sigma_{\xi}(f)\sigma_{\xi}(g)\big) z_{\xi} + \delta_{\xi}(f) g + \sigma_{\xi}(f)\delta_{\xi}(g)
		\]
		for all $f,g\in H_{\xi}^{o}$.
		We deduce that $\sigma_{\xi}\colon H_{\xi}^{o}\to H_{\xi}^{o}$ is an algebra automorphism and $\delta_{\xi}\colon H_{\xi}^{o}\to H_{\xi}^{o}$ is a $\sigma_{\xi}$-derivation. Therefore, $H_{\xi} = H_{\xi}^{o}[z_{\xi};\sigma_{\xi},\delta_{\xi}]$, as required.
	\end{proof}
	
	Another crucial observation, formulated as the proposition below, characterizes the expression of $\Delta(z_{V})$ for every word $V\in \PP(\Xi, <)$ under  our standing hypotheses on $\GG=(\{z_\xi\}_{\xi\in\Xi},<)$. A weaker statement for the graded setting was established in \cite[Lemma 3.3]{LZ25}.
	
	We now introduce additional notation. For words $V, W\in \PP(\Xi)$, write $W| V$ to mean that $W$ is a subword of $V$: if $V=\nu_1\cdots \nu_n$, then $W=\nu_{i_1}\cdots \nu_{i_s}$ for some indices $1\leq i_1<\cdots <i_s\leq n$. If $V= \xi_{1}^{p_{1}} \cdots \xi_{r}^{p_{r}}$ with pairwise distinct $\xi_1,\ldots, \xi_r \in \Xi$, every subword takes the form $W= \xi_{1}^{q_{1}} \cdots \xi_{r}^{q_{r}}$ with $q_i\leq p_i$. In this situation we write
	\[
	V/W:= \xi_{1}^{p_{1}-q_{1}} \cdots \xi_{r}^{p_{r} -q_{r}} \quad \text{and} \quad \binom{V}{W}:= \binom{p_1}{q_1}\cdots \binom{p_r}{q_r}.
	\]
	In particular, this notation applies whenever $V,W\in \PP(\Xi, <)$.
	
	\begin{prop}\label{comultiplication}
		Assume $\GG=(\{z_\xi\}_{\xi \in\Xi},<)$ is $\dd$-triangular and $(\Delta,\dd)$-dominated. Then
		\begin{eqnarray*}
			\Delta(z_{V}) & \in &  \sum_{W\mid V}\binom{V}{W}z_{W} \otimes z_{V/W} + \Span\{z_{P}\otimes z_{Q} \mid (P,Q) \in \mathcal{X}(V)\}
		\end{eqnarray*}
		for every word $V\in \PP(\Xi, <)$, where $\mathcal{X}(V) = \mathcal{X}_{0}(V)\cup \mathcal{X}_{1}(V) \subseteq  \PP_{+}(\Xi,<)^{2}$ is defined by
		\begin{eqnarray*}
			\mathcal{X}_{0}(V) &=& \big\{ (P,Q) \in \PP_{+}(\Xi,<)^{2} \,\big|\,  \dd(PQ) = \dd(V),\; P, Q, \Pi(PQ)<_{\rr} V \big\}\\
			\mathcal{X}_{1}(V) &=& \big\{ (P,Q) \in \PP_{+}(\Xi,<)^{2} \,\big|\, \dd(PQ) < \dd(V),\; P, Q, \Pi(PQ) <_{\rr,\lex}V \big\}.
		\end{eqnarray*}
	\end{prop}
	
	\begin{proof}
		We may assume $V\neq \emptyset$, as the statement is trivial for $V=\emptyset$.
		
		By Lemma \ref{rearrangement}, for each nonempty word $D\in \PP_{+}(\Xi)$ we may fix a finite subset
		$$\mathcal{Y}(D) \subseteq \{E\in \PP_{+}(\Xi, <) \mid  \dd(E) \leq \dd(D),\; E\leq _{\rr,\lex} \Pi(D)\}$$
		such that $z_{D} \in \Span \{ z_{E} \mid E\in \mathcal{Y}(D)\}.$ For $D\in \PP_{+}(\Xi,<)$ we simply take $\mathcal{Y}(D) =\{D\}$.
		
		We induct on the length $l(V)$ of $V$. If $l(V)=1$, the assertion follows directly from the hypotheses. Now suppose $l(V)>1$ and write $V=\xi U$ with $\xi \in \Xi$. By the induction hypothesis,
		\begin{eqnarray*}
			\Delta(z_{V}) &=&
			\left(1\otimes z_{\xi}+z_{\xi}\otimes1+\sum_{(M,N)\in \mathcal{X}(\xi)} a_{M,N} z_{M}\otimes z_{N} \right) \\
			&& \cdot\left(\sum_{E\mid U}\binom{U}{E}z_{E}\otimes z_{U/E}+ \sum_{(S,T)\in \mathcal{X}(U)} b_{S,T} z_{S}\otimes z_{T}\right),
		\end{eqnarray*}
		with $a_{M,N},b_{S,T} \in \Bbbk$, and $\mathcal{X}(\xi), \mathcal{X}(U) \subseteq \PP_{+}(\Xi,<)^{2}$ defined in the obvious way. Since $\xi$ is the minimal letter appearing in $V$,
		\[
		\sum_{W\mid V}\binom{V}{W}z_{W} \otimes z_{V/W} = \left(1\otimes z_{\xi}+z_{\xi}\otimes1\right) \left(\sum_{E\mid U}\binom{U}{E}z_{E}\otimes z_{U/E}\right) .
		\]
		Subtracting these two expressions shows that $\Delta(z_{V}) - \sum_{W\mid V}\binom{V}{W}z_{W} \otimes z_{V/W}$ lies in the span of tensors $z_{P} \otimes z_{Q}$ where $(P,Q)$ satisfies one of the following four cases:
		\begin{enumerate}
			\item $P=S$ and $Q\in \mathcal{Y}(\xi T)$ for some $(S,T) \in \mathcal{X}(U)$;
			\item  $P\in \mathcal{Y}(\xi S)$ and $Q=T$ for some $(S,T) \in \mathcal{X}(U)$;
			\item  $P = ME$ and $Q=N(U/E)$ for some $(M,N)\in \mathcal{X}(\xi)$ and $E\mid U$;
			\item $P\in \mathcal{Y}(MS)$ and $Q\in \mathcal{Y}(NT)$ for some $(M,N)\in \mathcal{X}(\xi)$ and $(S,T) \in \mathcal{X}(U)$.
		\end{enumerate}
		In Case (3), note that $ME, N(U/E)\in \PP_{+}(\Xi, <)$ because $M, N<_{\rr} \xi \leq _{\mathfrak{l},\lex} U$.
		
		It remains to verify that $(P,Q) \in \mathcal{X}(V)$ in every case. We treat Case (4) as representative. Clearly $\dd(PQ) \leq \dd (V)$. By Lemma \ref{order-compare-rearrangement},
		\[
		P ,Q <_{\rr,\lex} \Pi(PQ) \leq _{\rr,\lex}   \Pi(MSNT) = \Pi(MNST)  <_{\rr,\lex} \Pi(\xi U) =V.
		\]
		Suppose additionally that $\dd(PQ) =\dd(V)$. Then $\dd(MN) =\dd(\xi)$, $\dd(ST)=\dd(U)$, $\dd(P)=\dd(MS)$ and $\dd(Q)=\dd(NT)$. Since $M,N<_{\rr} \xi$ and  $S,T <_{\rr} U$, Lemma \ref{order-compare-rearrangement} yields
		\[
		P \leq _{\rr} \Pi(MS) <_{\rr} \Pi(\xi U) =V  \quad \text{and} \quad Q \leq _{\rr} \Pi(NT) <_{\rr} \Pi(\xi U) =V.
		\]
		Moreover, $\Pi(MN) <_{\rr} \xi$ and $\Pi(PQ) <_{\rr} U$, so
		\[
		\Pi(PQ) \leq _{\rr} \Pi(MSNT) = \Pi(MNST) <_{\rr} \Pi(\xi U) =V.
		\]
		Thus $(P,Q) \in \mathcal{X}(V)$. Similar, simpler arguments apply to the remaining three cases.
	\end{proof}
	
	\begin{corollary}\label{core facts}
		Let $\GG=(\{z_\xi\}_{\xi \in\Xi},<)$ be a system of PBW generators of $H$. Assume $\GG$ is $\dd$-triangular and $(\Delta,\dd)$-dominated. Let $U, V\in \PP(\Xi,<)$ satisfy $U|V$. Then
		\begin{eqnarray*}
			(\phi_{U}\otimes \id_{H})(\Delta(f))   &\in &  \binom{V}{U} z_{V/U} + \Span\{z_{N}\mid N\in \PP(\Xi,<),\; N<_{\rr,\dd} V/U \}\\
			(\id_{H}\otimes \phi_{U}) (\Delta(f))  &\in &  \binom{V}{U} z_{V/U} + \Span\{z_{N}\mid N\in \PP(\Xi,<),\; N<_{\rr,\dd} V/U \}
		\end{eqnarray*}
		for any $f\in  z_{V} +\Span\{z_{W} \mid W\in \PP(\Xi, <),\; W <_{\rr,\dd} V\};$ and
		\begin{eqnarray*}
			(\phi_{U}\otimes \id_{H})(\Delta(f))  &\in& \binom{V}{U} z_{V/U} + \Span\{z_{N}\mid N\in \PP(\Xi,<),\; \dd(N) \leq \dd(V/U),\; N <_{\rr,\lex} V/U  \}\\
			(\id_{H}\otimes \phi_{U}) (\Delta(f)) &\in&  \binom{V}{U} z_{V/U} + \Span\{z_{N}\mid N\in \PP(\Xi,<),\; \dd(N) \leq \dd(V/U),\; N <_{\rr,\lex} V/U  \}
		\end{eqnarray*}
		for any $f\in  z_{V} +\Span\{z_{W} \mid W\in \PP(\Xi, <),\; \dd(W) \leq \dd(V),\; W <_{\rr,\lex} V\}.$
	\end{corollary}
	
	\begin{proof}
		We treat $(\phi_{U}\otimes \id_{H})(\Delta(f))$; the case of $(\id_{H}\otimes \phi_{U}) (\Delta(f))$ is analogous.
		
		First apply Proposition \ref{comultiplication} to $z_{W}$ for $W\in \PP(\Xi, <)$ with $\dd(W) <\dd(V)$. One obtains
		\begin{eqnarray*}
			(\phi_{U}\otimes \id_{H}) (\Delta(z_{W})) &\in & \Span\{z_{Q}\mid Q\in \PP(\Xi,<),\; \dd(Q) < \dd(V/U) \}.
		\end{eqnarray*}
		
		Next apply Proposition \ref{comultiplication} to $z_{W}$ for $W \in \PP(\Xi, <)$ with $\dd(W) \leq \dd(V)$ and $W <_{\rr,\lex} V$. Consider any $(P,Q) \in \supp_{\GG} (\Delta(z_{W}))$. We have $\dd(PQ) \leq \dd(W)$ and $\Pi(PQ) \leq _{\rr,\lex} W$. By Lemma \ref{order-compare-rearrangement}, whenever $P=U$ we get $\dd(Q) \leq \dd(V/U)$ and $Q <_{\rr,\lex} V/U$. Therefore
		\begin{eqnarray*}
			(\phi_{U}\otimes \id_{H}) (\Delta(z_{W})) &\in & \Span\{z_{Q}\mid Q\in \PP(\Xi,<),\; \dd(Q) \leq  \dd(V/U),\; Q <_{\rr,\lex} V/U \}.
		\end{eqnarray*}
		
		Finally apply Proposition \ref{comultiplication} to $z_{V}$. Take any $(P,Q) \in \supp_{\GG} (\Delta(z_{V}))$ with $(P,Q) \neq (E, V/E)$ for any subword $E\mid V$. Then $\dd(PQ) \leq \dd(V)$ and $\Pi(PQ) <_{\rr,\lex} V$. By Lemma \ref{order-compare-rearrangement}, if $P=U$ then $\dd(Q) \leq \dd(V/U)$ and $Q <_{\rr,\lex} V/U$. Hence
		\begin{eqnarray*}
			(\phi_{U}\otimes \id_{H})(\Delta(z_{V})) &\in& \binom{V}{U} z_{V/U}  + \Span\{z_{Q}\mid Q\in \PP(\Xi,<),\; \dd(Q)\leq \dd(V/U),\; Q <_{\rr,\lex} V/U  \}.
		\end{eqnarray*}
		
		Combining these three estimates yields the claim.
	\end{proof}
	
	\subsection{Additional complementary results}
	
	\begin{prop}\label{change-of-order}
		Assume $\GG=(\{z_{\xi}\}_{\xi\in \Xi}, <)$ is $\dd$-triangular. For any total order $\triangleleft$ on $\Xi$: if $\{z_V\}_{V\in \PP(\Xi,<)}$ spans $H$ (resp.\ is linearly independent), then so does $\{z_V\}_{V\in \PP(\Xi,\triangleleft)}$. Consequently, if $\GG$ is a PBW generating system of $H$, then $(\{z_{\xi}\}_{\xi\in \Xi}, \triangleleft)$ is too.
	\end{prop}
	
	\begin{proof}
		Let $\triangleleft$ be a total order on $\Xi$. First observe that the rearrangement map $\Pi = \Pi_{<}\colon \PP(\Xi) \to \PP(\Xi, <)$ restricts to a bijection $\Pi\colon \PP(\Xi, \triangleleft) \to \PP(\Xi, <)$.
		
		Suppose $\{z_V\}_{V\in \PP(\Xi,<)}$ spans $H$. Let $A$ denote the subspace spanned by $\{z_V\}_{V\in \PP(\Xi,\triangleleft)}$. To show $A=H$, it suffices to verify $z_V\in A$ for every $V\in {\PP(\Xi,<)}$. We argue by induction on $V$ with respect to the well-order $<_{\rr,\dd}$. If $V=\emptyset$ is minimal, then $V\in \PP(\Xi,\triangleleft)$ and $z_{V} \in A$. If $V$ is not minimal, Lemma \ref{rearrangement} gives
		$$z_{\Pi^{-1}(V)} \in \Bbbk^{\times}\cdot z_{V} + \Span\{z_{W}\mid W\in\PP_{+}(\Xi,<),\; W<_{\rr,\dd} V\}.$$
		Rearranging,
		$$
		z_V \in \Bbbk^{\times}\cdot z_{\Pi^{-1}(V)} + \Span\{z_W\mid W \in \PP(\Xi,<),\; W <_{\rr,\dd} V \}.
		$$
		The induction hypothesis now implies $z_{V}\in A$ for all $V\in {\PP(\Xi,<)}$.
		
		Next suppose $\{z_V\}_{V\in\PP(\Xi,<)}$ is linearly independent. Assume for contradiction that $\{z_V\}_{V\in \PP(\Xi,\triangleleft)}$ is linearly dependent. Then there exist $n\geq 1$, distinct words $V_1,\dots,V_n \in \PP(\Xi,\triangleleft)$ and nonzero scalars $a_1,\dots,a_n \in \Bbbk^{\times}$ such that $\sum_{i=1}^{n}a_i z_{V_i}=0$. Without loss of generality, arrange the indices so that
		$$\Pi(V_1)<_{\rr,\dd}\Pi(V_2)<_{\rr,\dd} \cdots <_{\rr,\dd}\Pi(V_n).$$
		By Lemma \ref{rearrangement}, for each $i=1,\ldots, n$ there exists $b_{i}\in \Bbbk^{\times}$ with
		\begin{eqnarray*}
			z_{V_{i}} - b_{i} z_{\Pi(V_{i})} &\in& \Span\{z_W\mid W \in \PP(\Xi,<),\; W <_{\rr,\dd} \Pi(V_{i}) \} \\
			&\subseteq& \Span\{z_W\mid W \in \PP(\Xi,<),\; W <_{\rr,\dd} \Pi(V_{n}) \}.
		\end{eqnarray*}
		Substitution yields $z_{\Pi(V_n)}\in \Span\{ z_W\mid W\in\PP(\Xi,<),\; W <_{\rr,\dd} \Pi(V_n) \}$, contradicting linear independence.
	\end{proof}
	
	Recall that a \emph{filtration} on $H$ is a family of subspaces $\mathfrak{F}=\{H_{[m]}\}_{m\geq 0}$ satisfying:
	\begin{enumerate}
		\item $H_{[m]} \subseteq H_{[m+1]}$ for all $m\geq 0$;
		\item $H=\bigcup_{m\geq 0}H_{[m]}$ and $1_{H}\in H_{[0]}$;
		\item $H_{[m]} H_{[n]}\subseteq H_{[m+n]}$ for all $m,n\geq 0$.
	\end{enumerate}
	Given such a filtration, define the associated graded algebra
	$$\gr_{\mathfrak{F}}(H):=\bigoplus_{m\geq 0} H_{[m]}/H_{[m-1]}$$ with the convention $H_{[-1]}=0$. 
	Its multiplication is induced by $(x+H_{[l-1]})(y+H_{[m-1]}) =xy+H_{[l+m-1]}$ for $x\in H_{[l]}, y\in H_{[m]}$.
	
	\begin{lemma}\label{filtration-PBW}
		Assume $H$ admits a triangular PBW generating system $\GG=(\{z_{\xi}\}_{\xi\in \Xi}, <)$, where $\Xi = \{\xi_{1},\ldots, \xi_{n}\}$ is finite with $\xi_{1}<\cdots< \xi_{n}$. Suppose for all $1\leq i<j\leq n$,
		\[
		z_{\xi_{j}} z_{\xi_{i}} \in q_{ij} z_{\xi_{i}} z_{\xi_{j}}+ \Span\{z_{W} \mid W \in \PP_{+}(\Xi, <),\; W<_{\rr,\lex} \xi_{i}\xi_{j}\},\quad  q_{ij}\in \Bbbk^{\times}.
		\]
		Then there exist a filtration $\mathfrak{F}=\{H_{[m]}\}_{m\geq 0 }$ on $H$ and positive integers $d_{1},\ldots, d_{n}>0$ such that
		\[
		\gr_{\mathfrak{F}}(H) \,\cong\, \frac{\Bbbk\langle x_{1},\dots, x_{n} \mid d_{1},\dots, d_{n} \rangle}{\big(x_{j}x_{i} -q_{ij}x_{i}x_{j} \mid 1\leq i< j\leq n\big)}
		\]
		as $\mathbb{N}$-graded algebras. Here $\Bbbk\langle x_{1},\dots, x_{n} \mid d_{1},\dots, d_{n} \rangle$ denotes the $\mathbb{N}$-graded free algebra on $x_{1},\dots, x_{n}$, where each $x_i$ is assigned degree $d_i$.
	\end{lemma}
	
	\begin{proof}
		We may assume $n\geq 2$, since the case $n=1$ is trivial. By Lemma \ref{triangular-construct}, there exists a positive map $\dd\colon\Xi\to \mathbb{N}$ such that for all $1\leq i<j\leq n$,
		\[
		z_{\xi_{j}} z_{\xi_{i}} \in q_{ij} z_{\xi_{i}} z_{\xi_{j}} + \Span\{z_{W} \mid W \in \PP_{+}(\Xi, <),\; \dd(W) < \dd(\xi_{i}\xi_{j}),\; W<_{\rr,\lex} \xi_{i}\xi_{j} \}.
		\]
		For each integer $m\geq 0$, set
		\[
		H_{[m]} := \Span\{z_{V}\mid V\in \PP(\Xi, <),\; \dd(V) \leq m\}.
		\]
		By Lemma \ref{rearrangement}, $\mathfrak{F}=\{H_{[m]}\}_{m\geq 0 }$ is a filtration on $H$. For $i=1,\ldots, n$, let $d_{i} =\dd(z_{\xi_{i}})$ and write $\overline{z}_{\xi_{i}} = z_{\xi_{i}} + H_{[d_{i}-1]} \in \gr_{\mathfrak{F}}(H)$. One immediately checks $\overline{z}_{\xi_{j}} \overline{z}_{\xi_{i}} = q_{ij} \overline{z}_{\xi_{i}} \overline{z}_{\xi_{j}}$ for $1\leq i<j\leq n$, and $\overline{\GG}= (\{\overline{z}_{\xi}\}_{\xi\in \Xi}, <)$ is a homogeneous PBW generating system for $\gr_{\mathfrak{F}}(H)$. The algebra homomorphism $\Bbbk\langle x_{1},\ldots, x_{n} \mid d_{1},\ldots, d_{n} \rangle \to \gr_{\mathfrak{F}}(H)$ sending $x_{i} \mapsto \overline{z}_{\xi_{i}}$ therefore descends to the required isomorphism
		\[
		\frac{\Bbbk\langle x_{1},\dots, x_{n} \mid d_{1},\dots, d_{n} \rangle}{\big(x_{j}x_{i} -q_{ij}x_{i}x_{j} \mid 1\leq i< j\leq n\big)} \xrightarrow{\cong} \gr_{\mathfrak{F}}(H)
		\]
		of $\mathbb{N}$-graded algebras.
	\end{proof}
	
	We close the section by showing that algebras admitting such well-behaved PBW generating systems satisfy fundamental ring-theoretic and homological properties. All unexplained terminology is standard; references may be found for instance in \cite{Le92, RRZ14}.
	
	\begin{prop}\label{ring-homo-properties}
		Suppose $H$ admits a triangular PBW generating system $\GG=(\{z_{\xi}\}_{\xi\in \Xi}, <)$ with finite index set $\Xi$ of cardinality $n$. Then
		\begin{enumerate}
			\item $H$ is a Noetherian domain;
			\item $H$ has GK-dimension $n$, global dimension $n$, and Krull dimension $\leq n$;
			\item $H$ is Auslander regular, GK-Cohen-Macaulay and skew $n$-Calabi-Yau.
		\end{enumerate}
	\end{prop}
	
	\begin{proof}
		By Lemma \ref{filtration-PBW}, $H$ carries a filtration $\mathfrak{F}$ such that $\gr_{\mathfrak{F}}(H)$ satisfies all the listed properties.
		
		Part (1) follows from \cite[Theorem 1.6.6, Theorem 1.6.9]{MR01}. By \cite[Proposition 6.6]{KL00} and \cite[Lemma 6.5.6]{MR01}, $H$ has GK-dimension $n$ and Krull dimension $\leq n$. It is well known that a filtered algebra is Auslander regular, GK-Cohen-Macaulay and skew $n$-Calabi-Yau whenever its associated graded algebra has these properties, provided the associated graded algebra is locally finite; see \cite[Proposition 7.6]{ZSL19} and \cite[Theorem 4.4]{ZSL20-1}. This yields statement (3).
		
		It remains to verify that the global dimension of $H$ equals $n$. By \cite[Corollary 6.18]{MR01}, $\operatorname{gldim} H \leq n$. By Lemma \ref{rearrangement}, the linear map $\phi_{\emptyset}\colon H\to \Bbbk$ defined by $z_{V}\mapsto \delta_{\emptyset, V}$ is an algebra homomorphism, so $\Bbbk$ becomes a natural left $H$-module. Since $H$ is GK-Cohen-Macaulay,
		\[
		\GKdim(\Bbbk) + \min\{i\geq 0 \mid \operatorname{Ext}_{H}^{i} (\Bbbk, H) \neq 0\} = \GKdim (H) =n.
		\]
		Consequently $\operatorname{Ext}_{H}^{n} (\Bbbk, H) \neq  0$, hence $\operatorname{gldim} H = n$, as desired.
	\end{proof}

	\section{Thin replacements on  PBW generators}
	\label{sec:thin-replacement}
	
	This section develops the thin replacement method for PBW generators, a technique originated by Kharchenko \cite{Kh08} for studying coideal subalgebras of quantum  Borel algebras. This fundamental perturbation technique preserves all key structural properties of PBW generating systems introduced in the previous section, including $\dd$-triangularity and $(\Delta,\dd)$-domination. More importantly, this method serves as a  systematic tool to produce PBW generating systems for subalgebras from those of the  ambient algebra under mild assumptions. 
	
	We retain all notation and conventions established in the preceding sections. To streamline notation further, for every non-decreasing word $V\in \mathcal{P}(\Xi, <)$, define
	\begin{eqnarray*}
		H[\GG, V,\dd] &:=& \Span\big\{z_{W} \,\big|\, W \in \mathcal{P}(\Xi, <),\; \dd(W) \leq \dd(V),\; W<_{\rr,\lex} V \big\}; \\
		H_{+}[\GG, V,\dd] &:=& \Span\big\{z_{W} \,\big|\, W \in \mathcal{P}_{+}(\Xi, <),\; \dd(W) \leq \dd(V),\; W<_{\rr,\lex} V \big\}.
	\end{eqnarray*}
	
	\subsection{Definitions and basic properties}
	
	\begin{definition}\label{thin replacement}
		A \emph{$\dd$-thin replacement}
		of $\GG=(\{z_{\xi}\}_{\xi\in \Xi}, <)$ is an ordered family
		$\hat{\GG}=(\{\hat{z}_{\xi}\}_{\xi\in \Xi}, <)$ with the same index set and total order such that for every $\xi\in\Xi$,
		\[
		\hat{z}_{\xi}-z_{\xi}\in  H_{+}[\GG,\xi,\dd].
		\]
	\end{definition}
	
	\begin{lemma}\label{replacement-control}
		Let $\hat{\GG}=(\{\hat{z}_{\xi}\}_{\xi\in \Xi}, <)$ be a $\dd$-thin replacement of $\GG=(\{z_{\xi}\}_{\xi\in \Xi}, <)$. Assume $\GG$ is $\dd$-triangular. Then $\hat{z}_V -z_{V} \in H_{+}[\GG,V,\dd]$ for any non-decreasing word $V\in \mathcal{P}(\Xi,<)$.
	\end{lemma}
	
	\begin{proof}
		Set $f_{\xi} := \hat{z}_\xi - z_\xi$. By definition, $f_{\xi} \in  H_{+}[\GG,\xi,\dd]$. Write $V=\xi_1\cdots\xi_n$. We expand
		\[
		\hat{z}_V -z_{V} = \prod_{i=1}^{n}(z_{\xi_i}+f_{\xi_i})-\prod_{i=1}^{n}z_{\xi_{i}} = \sum_{I\subset \{1,\ldots, n\}} g_{I,1}\cdots g_{I,n},
		\]
		where $g_{I,j} =z_{\xi_{j}}$ if $j\in I$, and $g_{I,j} =f_{\xi_{j}}$ if $j\not\in I$. By Lemma \ref{rearrangement}, each product $g_{I,1}\cdots g_{I,n}$ lies in $H_{+}[\GG,V,\dd]$. The assertion follows immediately.
	\end{proof}
	
	\begin{lemma}\label{equality of H}
		Let $\hat{\GG}=(\{\hat{z}_{\xi}\}_{\xi\in \Xi}, <)$ be a $\dd$-thin replacement of $\GG=(\{z_{\xi}\}_{\xi\in \Xi}, <)$. Assume $\GG$ is $\dd$-triangular. Then for each non-decreasing word $V\in \mathcal{P}(\Xi, <)$,
		\[
		H_{+}[\GG,V,\dd]=H_{+}[\hat{\GG},V,\dd]
		\quad \text{and} \quad
		H[\GG,V,\dd]=H[\hat{\GG},V,\dd].
		\]
	\end{lemma}
	\begin{proof}
		Define a partial order $<_{\dd}$ on $\mathcal{P}(\Xi)$: for $D, E\in \mathcal{P}(\Xi)$, set $D<_{\dd} E$ if and only if $\dd(D) \leq \dd(E)$ and $D <_{\rr,\lex} E$. An argument analogous to Lemma \ref{wellorder} shows that $<_{\dd}$ is well-founded. We prove $H_{+}[\GG,V,\dd]=H_{+}[\hat{\GG},V,\dd]$ by induction on $V$ with respect to $<_{\dd}$.
		
		The equality trivially holds for the minimal element $V=\emptyset$. Now suppose $V\neq \emptyset$. By the induction hypothesis together with Lemma \ref{replacement-control}, for every $W <_{\dd} V$ we have
		\[
		\hat{z}_{W} - z_{W} \in H_{+}[\GG,W,\dd] = H_{+}[\hat{\GG},W,\dd].
		\]
		This implies
		\[
		\hat{z}_{W} \in z_{W}+ H_{+}[\GG,W,\dd] \subseteq H_{+}[\GG,V,\dd],\qquad
		z_{W} \in \hat{z}_{W} + H_{+}[\hat{\GG},W,\dd] \subseteq H_{+}[\hat{\GG},V,\dd].
		\]
		Hence $H_{+}[\GG,V,\dd]=H_{+}[\hat{\GG},V,\dd]$. Moreover,
		\[
		H[\GG,V,\dd] = H_{+}[\GG,V,\dd] + \Bbbk z_{\emptyset},\qquad
		H[\hat{\GG},V,\dd] = H_{+}[\hat{\GG},V,\dd] + \Bbbk z_{\emptyset},
		\]
		so the equality $H[\GG,V,\dd]=H[\hat{\GG},V,\dd]$ follows.
	\end{proof}
	
	\begin{prop}\label{replacement facts}
		Let $\hat{\GG}=(\{\hat{z}_{\xi}\}_{\xi\in \Xi}, <)$ be a $\dd$-thin replacement of $\GG=(\{z_{\xi}\}_{\xi\in \Xi}, <)$. Assume $\GG$ is $\dd$-triangular. If $\GG$ forms a system of PBW generators of $H$, then so does $\hat{\GG}$.
	\end{prop}
	
	\begin{proof}
		Suppose $\GG$ is a PBW generating system for $H$. From Lemma \ref{equality of H},
		\[
		H=\sum_{V\in \mathcal{P}(\Xi,<)} H[\GG, V, \dd] = \sum_{V\in \mathcal{P}(\Xi,<)} H[\hat{\GG}, V, \dd] = \Span \{\hat{z}_{W} \mid W\in \mathcal{P}(\Xi,<)\}.
		\]
		It remains to verify linear independence of $\{\hat{z}_{V}\}_{V\in\mathcal{P}(\Xi,<)}$. Assume for contradiction that $\sum_{i=1}^{m}c_{i} \hat{z}_{V_{i}}=0$ for some $m\geq 1$, distinct words $V_{1},\dots,V_{m} \in \mathcal{P}(\Xi,<)$ and nonzero scalars $c_{1},\dots,c_{m} \in \Bbbk^{\times}$. We may order the indices such that
		\[
		V_{1}<_{\rr,\dd}V_{2}<_{\rr,\dd} \cdots <_{\rr,\dd}V_{m}.
		\]
		Using Lemma \ref{replacement-control}, each $\hat{z}_{V_i}=z_{V_i}+\sum_{U<_{\rr,\dd}V_i}c_{i,U}z_{U}$, hence
		\[
		0=\sum_{i=1}^{m} c_{i} \hat{z}_{V_{i}}
		= c_{m} z_{V_{m}} + \sum_{P<_{\rr,\dd}V_{m}} d_{P}\,z_{P}
		\]
		for suitable scalars $c_{i,U},d_{P} \in \Bbbk$. This contradicts linear independence of $\{z_{V}\}_{V\in\mathcal{P}(\Xi,<)}$.
	\end{proof}
	
	The following proposition shows that the structural conditions for PBW generating systems introduced in the previous section are preserved under thin replacement.
	
	\begin{prop}\label{proposition of replacement}
		Let $\hat{\GG}=(\{\hat{z}_{\xi}\}_{\xi\in \Xi}, <)$ be a $\dd$-thin replacement of  $\GG=(\{z_{\xi}\}_{\xi\in \Xi}, <)$.
		\begin{enumerate}
			\item If $\GG$ is $\dd$-triangular, then so is $\hat{\GG}$;
			\item If $\GG$ is $\dd$-triangular and $(\Delta,\dd)$-dominated, then so is $\hat{\GG}$.
		\end{enumerate}
	\end{prop}
	
	\begin{proof}
		(1) Assume $\GG$ is $\dd$-triangular. For any pair $\eta<\xi\in\Xi$, there exists $a\in \Bbbk^{\times}$ with
		\[
		z_\xi z_\eta \in a\cdot z_\eta z_\xi + H_{+}[\GG,\eta\xi,\dd].
		\]
		Write $f_{\xi} = \hat{z}_{\xi} -z_{\xi}$ and $f_{\eta} = \hat{z}_{\eta} -z_{\eta}$. Expanding,
		\[
		\begin{aligned}
			\hat{z}_\xi\hat{z}_\eta - a\cdot\hat{z}_\eta\hat{z}_\xi
			&= (z_\xi+f_\xi)(z_\eta+f_\eta) - a\,(z_\eta+f_\eta)(z_\xi+f_\xi) \\
			&= (z_\xi z_\eta - a\,z_\eta z_\xi)
			+ z_\xi f_\eta + f_\xi z_\eta + f_\xi f_\eta
			- a\,z_\eta f_\xi - a\,f_\eta z_\xi - a\,f_\eta f_\xi \\
			&\in z_\xi f_\eta + f_\xi z_\eta + f_\xi f_\eta
			- a\,z_\eta f_\xi - a\,f_\eta z_\xi - a\,f_\eta f_\xi + H_{+}[\GG,\eta\xi,\dd].
		\end{aligned}
		\]
		We claim each term $z_\xi f_\eta, f_\xi z_\eta, z_\eta f_\xi, f_\eta z_\xi, f_\xi f_\eta, f_\eta f_\xi$ belongs to $H_{+}[\GG,\eta\xi,\dd]$. We verify one representative case. Write
		$$f_\xi=\sum_{j=1}^{s} c_j z_{W_j},$$
		where $c_{j}\in \Bbbk$, $W_{j} \in \mathcal{P}_{+}(\Xi, <)$ satisfy $\dd(W_{j}) \leq \dd(\xi)$ and $W_{j} <_{\rr,\lex} \xi$. By Lemma \ref{rearrangement},
		\[
		z_\eta f_\xi = \sum_j c_j z_\eta z_{W_j}
		\in \sum_{j} \bigl(\Bbbk^\times\!\cdot\! z_{\Pi(\eta W_j)}
		+ H_{+}[\GG,\Pi(\eta W_j),\dd]\bigr)
		\subseteq H_{+}[\GG,\eta\xi,\dd].
		\]
		Identical reasoning applies to the remaining terms. Using Lemma \ref{equality of H},
		\[
		\hat{z}_\xi\hat{z}_\eta \in a\cdot\hat{z}_\eta\hat{z}_\xi
		+ H_{+}[\GG,\eta\xi,\dd]= a\cdot\hat{z}_\eta\hat{z}_\xi
		+ H_{+}[\hat{\GG},\eta\xi,\dd].
		\]
		Thus $\hat{\GG}$ is $\dd$-triangular.
		
		(2) Assume $\GG$ is $\dd$-triangular and $(\Delta,\dd)$-dominated. Part (1) already yields $\dd$-triangularity of $\hat{\GG}$, so we only need to check $(\Delta,\dd)$-domination. Fix $\xi\in \Xi$. By definition of $\dd$-thin replacement,
		\[
		\hat{z}_{\xi} =z_{\xi} + \sum_{j}c_{j} z_{W_{j}},
		\]
		with $c_{j}\in \Bbbk$, $W_{j} \in \mathcal{P}_{+}(\Xi, <)$, $\dd(W_{j}) \leq \dd(\xi)$ and $W_{j} <_{\rr,\lex} \xi$. Applying Proposition \ref{comultiplication},
		\begin{eqnarray*}
			\Delta(\hat{z}_{\xi}) & = & \Delta(z_{\xi}) + \sum_{j}c_{j} \Delta(z_{W_{j}}) \\
			& = & 1\otimes z_\xi + z_\xi\otimes 1 + \sum_{i}a_{i}\,z_{M_{i}}\otimes z_{N_{i}} \\
			&& + \sum_{j} c_{j} \Big( \sum_{E\mid W_{j}}  \binom{W_{j}}{E} z_{E} \otimes z_{W_{j}/E} + \sum_{k_{j}} b_{k_{j}} z_{P_{k_{j}}} \otimes z_{Q_{k_{j}}} \Big) \\
			& =& 1\otimes \hat{z}_\xi + \hat{z}_\xi\otimes 1 +  \sum_{i}a_{i}\,z_{M_{i}}\otimes z_{N_{i}} \\
			&& + \sum_{j} c_{j}  \Big(  \sum_{\substack{E\mid W_{j} \\ E\neq \emptyset, E\neq W_{j}}} \binom{W_{j}}{E} z_{E} \otimes z_{W_{j}/E} + \sum_{k_{j}} b_{k_{j}} z_{P_{k_{j}}} \otimes z_{Q_{k_{j}}} \Big),
		\end{eqnarray*}
		where $a_{i}, b_{k_{j}} \in \Bbbk$, $M_{i},N_{i}, P_{k_{j}}, Q_{k_{j}}\in\mathcal{P}_{+}(\Xi,<)$ satisfy
		$\dd(M_{i} N_{i}) \leq \dd(\xi)$, $M_{i}, N_{i}<_{\rr}\xi$,
		$\dd(P_{k_{j}}Q_{k_{j}}) \leq \dd(W_{j})$ and $P_{k_{j}}, Q_{k_{j}} <_{\rr,\lex} W_{j}$.
		
		From Lemma \ref{replacement-control} and Lemma \ref{equality of H}, for every $T\in \mathcal{P}(\Xi, <)$ we have
		\[
		z_{T} \in \hat{z}_{T} + H_{+}[\hat{\GG}, T, \dd].
		\]
		Substitute this relation for $T= M_{i}, N_{i}, E, W_{j}/E, P_{k_{j}}, Q_{k_{j}}$. It follows that
		\[
		\Delta(\hat{z}_{\xi}) \in 1\otimes \hat{z}_\xi + \hat{z}_\xi\otimes 1 + \Span \big \{\hat{z}_{P}\otimes \hat{z}_{Q} \,\big|\, P,Q\in \mathcal{P}_{+}(\Xi, <),\; \dd(PQ) \leq \dd(\xi),\; P, Q <_{\rr} \xi \big\}.
		\]
		Hence $\hat{\GG}$ is $(\Delta,\dd)$-dominated.
	\end{proof}
	
	\subsection{Construction of PBW generators for subalgebras}
	
	In this subsection, under suitable hypotheses we use thin replacements to construct PBW generating systems for subalgebras starting from a PBW generating system  of the ambient algebra.
	
	Let $A$ be a subalgebra of $H$. Define the subset $\Xi(A,\GG,\dd)\subseteq\Xi$ via
	\begin{eqnarray}\label{sub-index-set}
		\Xi(A,\GG,\dd):=\{\,\xi\in\Xi\mid (z_{\xi}+H[\GG,\xi,\dd]) \cap  A \neq \emptyset \,\}.
	\end{eqnarray}
	Note that $\xi \in \Xi(A,\GG,\dd)$ is equivalent to $(z_{\xi}+H_{+}[\GG,\xi,\dd]) \cap  A \neq \emptyset$, since $z_{\emptyset}=1_{H} \in A$. Observe that we may construct a $\dd$-thin replacement $\hat{\GG}=(\{\hat{z}_{\xi}\}_{\xi \in \Xi}, <)$ of $\GG$ satisfying $\hat{z}_{\xi} \in A$ for all $\xi\in \Xi(A,\GG,\dd)$. Indeed, for each $\xi\in \Xi(A,\GG,\dd)$ pick any element
	\[
	\hat{z}_{\xi}\in (z_{\xi}+H_{+}[\GG,\xi,\dd]) \cap  A;
	\]
	for all $\xi\notin \Xi(A,\GG,\dd)$, simply set $\hat{z}_{\xi}=z_{\xi}$.
	
	\begin{definition}
		Let $(\Xi,<)$ be a totally ordered set. A map $\dd\colon\Xi\to \Gamma$ is called \emph{order-preserving} if $\dd(\eta)\leq \dd(\xi)$ whenever $\eta< \xi$.
	\end{definition}
	
	\begin{thm}\label{PBW-sub-hereditary}
		Let $\GG=(\{z_{\xi}\}_{\xi\in \Xi}, <)$ be a PBW generating system of $H$.  Let $A$ be a subalgebra of $H$ such that  either $\Delta(A)\subseteq A\otimes H$ or $\Delta(A)\subseteq H\otimes A$.  Assume $\GG$ is $\dd$-triangular and $(\Delta,\dd)$-dominated for some order-preserving positive map $\dd\colon\Xi\to \Gamma$. 
		If $\hat{\mathbb{G}}=(\{\hat{z}_{\xi}\}_{\xi\in \Xi},<)$ is any $\dd$-thin replacement of $\mathbb{G}$ satisfying $\hat{z}_{\xi}\in A$ for all $\xi\in \Xi(A,\mathbb{G},\dd)$,
		then the restricted ordered family
		\[
		\hat{\mathbb{G}}\big|_{\Xi(A,\mathbb{G},\dd)}
		=\bigl(\{\hat{z}_{\xi}\}_{\xi\in \Xi(A,\mathbb{G},\dd)},\,{<}\big|_{\Xi(A,\mathbb{G},\dd)}\bigr)
		\]
		constitutes a PBW generating system for $A$ and is $\dd\big|_{\Xi(A,\mathbb{G},\dd)}$-triangular.
	\end{thm}
	
	The proof relies on the following technical lemma.
	
	\begin{lemma}\label{index-set-induced}
		Let $\GG=(\{z_{\xi}\}_{\xi\in \Xi}, <)$ be a PBW generating system of $H$. Let $A$ be a subalgebra of $H$ satisfying either $\Delta(A)\subseteq A\otimes H$ or $\Delta(A)\subseteq H\otimes A$. Suppose $\GG$ is $\dd$-triangular and $(\Delta,\dd)$-dominated for some order-preserving positive map $\dd\colon\Xi\to \Gamma$.  If for some word
		$V\in \mathcal{P}_{+}(\Xi, <)$,
		\[
		\big ( z_{V} + \Span\{z_{W} \mid W\in \mathcal{P}(\Xi, <),\; W<_{\rr,\dd} V\} \big) \cap A\neq \emptyset,
		\]
		then  $V\in \mathcal{P}_{+}\big(\Xi(A,\GG,\dd), {<}|_{\Xi(A,\GG,\dd)}\big)$.
	\end{lemma}
	
	\begin{proof}
		We treat the case $\Delta(A) \subseteq H\otimes A$; the alternative case follows symmetrically. We proceed by induction on the length of $V$. Write $\Xi':=\Xi(A,\GG,\dd)$ for brevity.
		
		First suppose $V=\xi$ has length one. By hypothesis, there exists $g_{\xi}$ such that $z_{\xi} +g_{\xi} \in A$ and
		$g_{\xi}\in \Span\{z_{W} \mid W\in \mathcal{P}(\Xi, <),\; W<_{\rr,\dd} \xi\}.$
		Since $\dd$ is order-preserving,
		\[
		W<_{\rr,\dd} \xi \quad \Longrightarrow \quad \dd(W) \leq \dd(\xi) \quad \text{and} \quad W<_{\rr,\lex} \xi.
		\]
		Hence $g_{\xi}\in H[\GG, \xi, \dd]$, which yields $\xi\in \Xi'$.
		
		Now let $V=V_{1} \xi$ have length greater than one. By assumption, there exists $$g_{V}\in \Span\{z_{W} \mid W\in \mathcal{P}(\Xi, <),\; W<_{\rr,\dd} V\}$$
		with $f_{V}=z_{V} +g_{V} \in A$.
		Set $\lambda_{V}:=\binom{V}{V_{1}}^{-1}$. By Corollary \ref{core facts},
		\[
		\big(\phi_{V_{1}} \otimes \id_{H} \big) \big (\Delta(\lambda_{V} f_{V})\big) \in \Big(  z_{\xi} + \Span \{z_{W} \mid W\in \mathcal{P}(\Xi, <),\; W<_{\rr,\dd} \xi\} \Big) \cap A.
		\]
		From the length-one case, $\xi \in \Xi'$. Also by Corollary \ref{core facts},
		\[
		\big(\phi_{\xi} \otimes \id_{H} \big) \big (\Delta(\lambda_{V} f_{V})\big) \in  \Big(  z_{V_{1}} + \Span \{z_{W} \mid W\in \mathcal{P}(\Xi, <),\; W<_{\rr,\dd} V_{1}\} \Big) \cap A.
		\]
		The induction hypothesis gives $V_{1}\in \mathcal{P}_{+}(\Xi', {<}|_{\Xi'})$. Therefore $V\in \mathcal{P}_{+}(\Xi', {<}|_{\Xi'})$.
	\end{proof}
	
	\begin{proof}[Proof of Theorem \ref{PBW-sub-hereditary}]
		Simplify notation by setting $\Xi'=\Xi(A,\GG,\dd)$ and $<'={<}|_{\Xi(A,\GG,\dd)}$.
		
		From Proposition \ref{replacement facts}, the family $\{\hat{z}_{V}\}_{V\in \mathcal{P}(\Xi', <')}$ is linearly independent. Let $$B=\Span\{\hat{z}_{V}\mid V\in \mathcal{P}(\Xi', <')\}\subseteq A.$$ We prove $f\in B$ for any nonzero element $f\in A$ by induction on $\operatorname{lw}_{\GG}(f)$ with respect to $<_{\rr,\dd}$, where $\operatorname{lw}_{\GG}(f)$ denotes the maximal element of $\supp_{\GG}(f)$ under $<_{\rr,\dd}$.
		
		If $\operatorname{lw}_{\GG}(f)=\emptyset$, then $f=a\cdot 1_{H}$ for some $a\in \Bbbk$, hence $f\in B$. Now suppose $\operatorname{lw}_{\GG}(f)= V \neq \emptyset$. Then there exists $a\in \Bbbk^{\times}$ such that
		\[
		a f\in \big ( z_{V} + \Span\{z_{W} \mid W\in \mathcal{P}(\Xi, <),\; W<_{\rr,\dd} V\} \big) \cap A.
		\]
		Lemma \ref{index-set-induced} implies $V\in \mathcal{P}_{+}(\Xi', <')$, so $\hat{z}_{V}\in B$. By Lemma \ref{replacement-control},
		\[
		af - \hat{z}_{V} =af-z_{V} +z_{V}- \hat{z}_{V} \in \big ( \Span\{z_{W} \mid W\in \mathcal{P}(\Xi, <),\; W<_{\rr,\dd} V\} \big) \cap A.
		\]
		If $af - \hat{z}_{V}=0$, then $f= a^{-1}\hat{z}_{V}\in B$. Otherwise, $\operatorname{lw}_{\GG}(af - \hat{z}_{V})  <_{\rr,\dd} \operatorname{lw}_{\GG}(f)$. The induction hypothesis yields $af - \hat{z}_{V} \in B$, hence $f\in B$. We conclude that $\{\hat{z}_{V}\}_{V\in \mathcal{P}(\Xi', <')}$ spans $A$. Therefore $\hat{\GG}|_{\Xi'}$ is a PBW generating system for $A$.
		
		It remains to establish that $\hat{\GG}|_{\Xi'}$ is $\dd|_{\Xi'}$-triangular. By Proposition \ref{proposition of replacement}(1), for any $\eta<\xi\in\Xi'$,
		\[
		\hat{z}_\xi\hat{z}_\eta \in  \Bbbk^{\times}\cdot\hat{z}_\eta\hat{z}_\xi
		+ H_{+}[\hat{\GG},\eta\xi,\dd] \cap A.
		\]
		Every element of $A$ has $\hat{\GG}$-support contained in $\mathcal{P}(\Xi',<')$. Consequently,
		\begin{eqnarray*}
			H_{+}[\hat{\GG},\eta\xi,\dd] \cap A
			&=& \Span\{\hat{z}_{W}\mid W\in \mathcal{P}_{+}(\Xi', <'),\; \dd(W) \leq \dd(\eta\xi),\; W<_{\rr,\lex} \eta\xi\} \\
			&=& A_{+}[\hat{\GG}|_{\Xi'},\eta\xi,\dd|_{\Xi'}].
		\end{eqnarray*}
		The second equality holds by definition. Thus
		\[
		\hat{z}_\xi\hat{z}_\eta \in  \Bbbk^{\times}\cdot\hat{z}_\eta\hat{z}_\xi + A_{+}[\hat{\GG}|_{\Xi'},\eta\xi,\dd|_{\Xi'}]
		\]
		for all $\eta<\xi\in\Xi'$, so $\hat{\GG}|_{\Xi'}$ is $\dd|_{\Xi'}$-triangular.
	\end{proof}

	\section{Iterated Hopf Ore extensions}
	\label{sec:IHOE-structure}
	
	Iterated Hopf Ore extensions (IHOEs) of $\Bbbk$ form a fundamental class of connected Hopf algebras with finite GK-dimension. In this section, we revisit the definition of IHOEs and explore their intrinsic structure leveraging the full machinery of PBW generating systems and the thin replacement method established in previous sections. Our main result demonstrates that every Hopf subalgebra and every quotient Hopf algebra of an IHOE of $\Bbbk$ is again an IHOE of $\Bbbk$. We further prove that left and right coideal subalgebras of an IHOE of $\Bbbk$ admit decompositions into ascending chains of subalgebras of the same type, where each inclusion corresponds to an Ore extension. Complementing these structural investigations, we introduce minimal PBW generating systems to give an explicit description of coradical filtrations, alongside a sharp criterion determining when an IHOE of $\Bbbk$ is an IHOE of its primitive part.
	
	Let us recall standard notions and notation for Hopf algebras. For any Hopf algebra $H$, we write $\Delta_{H}$ for its comultiplication and $\varepsilon_{H}$ for its counit. Set $H^{+} =\ker(\varepsilon_{H})$. The space of primitive elements of $H$ is denoted by $P(H)$. The \emph{coradical} of $H$, written $H_{(0)}$, is the sum of all simple subcoalgebras of $H$. A Hopf algebra is \emph{connected} if its coradical $H_{(0)}$ is one-dimensional. For each integer $n\geq 0$, the $n$-th layer of the \emph{coradical filtration} of $H$ is denoted by $H_{(n)}$. Whenever $H_{(0)}$ is a Hopf subalgebra of $H$ (in particular, if $H$ is connected), the coradical filtration is a Hopf algebra filtration. In this case, the associated $\mathbb{N}$-graded coalgebra
	\[
	\gr_{c}(H) = \bigoplus_{n\geq 0}H_{(n)}/H_{(n-1)},
	\]
	carries the structure of an $\mathbb{N}$-graded Hopf algebra, with the convention $H_{(-1)}=0$.
	
	\subsection{Definitions and basic properties}

	In Definition \ref{Ore extension}, we introduced Ore extensions and iterated Ore extensions of algebras. These constructions naturally extend to Hopf algebras, which motivates the following definition.
	
	\begin{definition}\label{Hopf Ore extension}
		Let $R$ be a Hopf algebra.
		\begin{enumerate}
			\item A \emph{Hopf Ore extension} of $R$ is a Hopf algebra $H$ containing $R$ as a Hopf subalgebra such that $H=R[X;\sigma,\delta]$ is an Ore extension of $R$ at the level of underlying algebras.
			
			\item An \emph{$n$-step iterated Hopf Ore extension} (\emph{$n$-step IHOE}) of $R$, where $n\geq 0$ is an integer, is a Hopf algebra $H$ equipped with an ascending chain of Hopf subalgebras
			\[
			R=H_{0} \subset H_{1} \subset \cdots \subset H_{n} =H
			\]
			such that $H_{i}$ is an Ore extension of $H_{i-1}$ for each $i=1,\dots, n$. The notation
			\[
			H=R[X_{1};\sigma_{1},\delta_{1}]\cdots [X_{n};\sigma_{n},\delta_{n}]
			\]
			indicates that $H_{0}=R$, $H_{n}=H$, and $H_{i}= H_{i-1}[X_{i};\sigma_{i},\delta_{i}]$ for all $i=1,\dots,n$. 
		\end{enumerate}
		Note that the $0$-step IHOE of $R$ is $R$ itself, and a $1$-step IHOE coincides with a Hopf Ore extension of $R$. We say that $H$ is an \emph{IHOE}  of $R$ if $H$ is an $n$-step IHOE of $R$ for some integer $n\geq 0$.
	\end{definition}
	
	Graded Hopf Ore extensions and ($n$-step) graded iterated Hopf Ore extensions  of graded Hopf algebras can be defined analogously.
	
	\begin{remark}\label{HOE-definition-compare}
		The definition of Hopf Ore extensions presented above is substantially weaker than the original definition in \cite[Definition 1.0]{Pa03} and the revised formulation from \cite[Definition 2.1]{BOZZ15}. Those two earlier versions impose extra conditions on $\Delta_{H}(X)$. As established in \cite[Proposition 2.8]{BOZZ15}, the revised formulation therein coincides with the definition above provided $R$ is connected. Moreover, Huang \cite{Hu20} proved that the formulation \cite[Definition 2.1]{BOZZ15} is equivalent to the version above whenever $R$ is noetherian and $R\otimes R$ is a domain.
	\end{remark}
	
	\begin{prop}\label{IHOE properties}
		Let $H = \Bbbk[X_{1};\sigma_{1},\delta_{1}]\cdots[X_{n};\sigma_{n},\delta_{n}]$ be an IHOE of $\Bbbk$.
		\begin{enumerate}
			\item $H$ is a noetherian connected Hopf algebra of GK-dimension $n$.
			\item For $i=1,\ldots, n$, one has
			\begin{eqnarray}\label{comultiplication-PBW}
				\Delta_{H}(X_{i}) \in  1\otimes X_{i} + X_{i}\otimes1+ H_{i-1} \otimes H_{i-1}.
			\end{eqnarray}
			\item For $i=1,\dots,n$, there exists a character $\chi_{i}:H_{i-1}\to\Bbbk$ such that
			\[
			\sigma_{i} (f) = \sum \chi_{i}(f_{(1)}) f_{(2)} = \sum f_{(1)} \chi_{i}(f_{(2)}), \quad f \in H_{i-1}.
			\]
			\item For $1\leq i<j\leq n$, one has $\sigma_{j} (X_{i})\in X_{i} +H_{i-1}$ and
			\begin{eqnarray}\label{permutation-PBW}
				X_{j}X_{i} \in  X_{i}X_{j} + H_{i-1} X_{j} + H_{j-1}.
			\end{eqnarray}
			\item $\GG=(\{z_{i}\}_{i\in \mathbb{I}_{n}}, <)$ is a triangular and $\Delta_{H}$-dominated PBW generating system of $H$, where $z_{i} = X_{i} -\varepsilon_{H}(X_{i})$ and $<$ is the natural total order on $\mathbb{I}_{n}$ given by $1<\cdots < n$.
		\end{enumerate}
	\end{prop}
	
	\begin{proof}
		Parts (1)--(4) are merely a reformulation of \cite[Proposition 2.2]{BZ22} and portions of \cite[Theorem 3.2]{BOZZ15}. One subtle point requires clarification. According to \cite[Proposition 2.2]{BZ22}, a change of variables for $X_{1},\dots,X_{n}$ may be necessary to ensure validity of relation \eqref{comultiplication-PBW}. In contrast, \cite[Proposition 2.8]{BOZZ15} claims that no such coordinate transformation is needed.
		
		We now turn to the proof of Part (5). By Lemma \ref{IOE facts}, $\GG$ forms a PBW generating system of $H$. Using Parts (2) and (4), one readily verifies that
		\begin{eqnarray*}
			& \Delta_{H}(z_{i}) \in  1\otimes z_{i} + z_{i}\otimes1+ (H_{i-1})^{+} \otimes (H_{i-1})^{+},    & \quad i=1,\ldots, n; \\
			& z_{j}z_{i} \in  z_{i}z_{j} + (H_{i-1}) z_{j} + (H_{j-1})^{+}, & \quad 1\leq i<j\leq n.
		\end{eqnarray*}
		Moreover,
		\begin{eqnarray*}
			&(H_{i-1})^{+} =\Span\{z_{V}\mid V\in \mathcal{P}_{+}(\mathbb{I}_{n}, <), ~ V<_{\rr} i\}, &  \quad i=1,\ldots, n,
		\end{eqnarray*}
		which implies that $\GG$ is triangular and $\Delta_{H}$-dominated.
	\end{proof}
	
	The next result can be regarded as a converse of Proposition \ref{IHOE properties} (5).
	
	\begin{prop}\label{PBW-imply-IHOE}
		Suppose a Hopf algebra $H$ possesses a finite PBW generating system $\GG=(\{z_{\xi}\}_{\xi\in \Xi}, <)$ that is triangular and $\Delta_{H}$-dominated. Let $\Xi=\{\xi_{1},\ldots, \xi_{n}\}$ be ordered so that $\xi_{1} < \cdots < \xi_{n}$. Then $H$ is an $n$-step IHOE of $\Bbbk$ with presentation
		\[
		H=\Bbbk[z_{\xi_{1}};\sigma_{1},\delta_{1}]\cdots [z_{\xi_{n}};\sigma_{n},\delta_{n}].
		\]
	\end{prop}
	
	\begin{proof}
		For each $i=1,\ldots, n$, define $\Xi_{i}= \{\xi_{1},\ldots, \xi_{i}\}$ and
		\[
		H_{i} = \Span\{z_{V} \mid V\in \mathcal{P}(\Xi_{i}, {<}|_{\Xi_{i}})\}.
		\]
		Since $\GG$ is $\Delta_{H}$-dominated,
		\[
		\Delta_{H}(z_{\xi_{i}}) \in 1\otimes z_{\xi_{i}} + z_{\xi_{i}} \otimes 1 + H_{i-1} \otimes H_{i-1},\quad i=1,\ldots, n.
		\]
		It follows that each $H_{i}$ is a Hopf subalgebra of $H$.
		By Lemma \ref{triangular-construct} and Proposition \ref{PBW-imply-IOE},
		\[
		\Bbbk=:H_{0} \subset H_{1} \subset \cdots \subset H_{n} =H
		\]
		forms an ascending chain of subalgebras of $H$, and each $H_{i} = H_{i-1}[z_{\xi_{i}};\sigma_{i},\delta_{i}]$ is an Ore extension of $H_{i-1}$. Hence $H$ is an IHOE of $\Bbbk$ of the asserted form.
	\end{proof}
	
	By Proposition \ref{IHOE properties}(1), every IHOE of $\Bbbk$ is a connected Hopf algebra of finite GK-dimension. The converse does not hold. Standard counterexamples are the universal enveloping algebras of finite-dimensional simple Lie algebras of dimension at least $4$. More sophisticated noncocommutative counterexamples can be found in \cite{HW26}. Nevertheless, classification results from \cite{Zh13, WZZ15} imply that every connected Hopf algebra with GK-dimension strictly less than $5$ is an IHOE of $\Bbbk$. In her PhD thesis \cite{Hu26}, Hu extended this result to all connected Hopf algebras with GK-dimension strictly less than $7$, without classifying them. The theorem below provides a natural class of connected Hopf algebras, each of which is realized as an IHOE of $\Bbbk$.
	
	\begin{thm}(\cite[Theorem B]{ZSL20})\label{Graded-imply-IHOE}
		If $H=\bigoplus_{\gamma\in\Gamma} H_{\gamma}$ is a connected $\Gamma$-graded Hopf algebra of finite GK-dimension $n$, then $H$ is an $n$-step graded IHOE of $\Bbbk$.
	\end{thm}
	
	\begin{proof}
		We revisit the proof of this theorem within the framework of the present paper. By Proposition \ref{Graded-imply-PBW}, $H$ admits a homogeneous PBW generating system $\GG=(\{z_{\xi}\}_{\xi\in \Xi}, <)$ that is $\dd$-triangular and $(\Delta_{H},\dd)$-dominated. Here $\dd\colon \Xi\to \Gamma$ denotes the positive map defined by $\dd(\xi) =\deg(z_{\xi})$. Lemma \ref{GKdimension} yields $n=\GKdim H$. Applying the graded variant of Proposition \ref{PBW-imply-IOE}, the reasoning from Proposition \ref{PBW-imply-IHOE} carries over to show that $H$ is an $n$-step graded IHOE of $\Bbbk$.
	\end{proof}
	
	\subsection{Hopf/coideal subalgebras of IHOEs}
	
	Within this subsection we investigate structural properties of Hopf/coideal subalgebras inside IHOEs. The following theorem constitutes one of the main results of this paper.
	
	\begin{thm}\label{Hopf-subalgebra-structure}
		Let $H$ be an IHOE of $\Bbbk$. Then every Hopf  (resp.\ left coideal, resp.\ right coideal) subalgebra of $H$ is an IOE of $\Bbbk$. More precisely, there exists a chain
		\[
		\Bbbk =A_{0} \subset A_{1} \subset \cdots \subset A_{m} =A
		\]
		of Hopf  (resp.\ left coideal, resp.\ right coideal) subalgebra of $H$ with $m=\operatorname{GKdim} A$, such that each $A_i$ is an Ore extension of $A_{i-1}$ for $1\le i\le m$.
	\end{thm}
	
	\begin{proof}
		We may assume $H\neq \Bbbk$. By Proposition \ref{IHOE properties}(5), $H$ admits a triangular, $\Delta_{H}$-dominated PBW generating system $\GG=(\{z_{i}\}_{i\in \II_{n}}, <)$ for some $n\geq 1$, where $<$ denotes the natural total order on $\II_{n}$ defined by $1<\cdots < n$. Then by Lemma \ref{triangular-construct}, we may choose an order-preserving positive map $\dd\colon\II_{n}\to\mathbb{N}$ such that $\GG$ is $\dd$-triangular and $(\Delta_{H},\dd)$-dominated. Define
		\[
		\II_{n}'=\II_{n}(A,\GG,\dd):=\{\, i\in\II_{n}\mid (z_{i}+H[\GG,i,\dd]) \cap A \neq \emptyset\}.
		\]
		We can select a $\dd$-thin replacement $\hat{\GG}=(\{\hat{z}_{i}\}_{i\in \II_{n}}, <)$ of $\GG$ satisfying $\hat{z}_{i} \in A$ for all $i\in \II_{n}'$. Theorem \ref{PBW-sub-hereditary} then implies that
		\[
		\hat{\GG}|_{\II_{n}'}:= (\{\hat{z}_{i}\}_{i\in \II_{n}'}, {<}|_{\II_{n}'})
		\]
		forms a $\dd|_{\II_{n}'}$-triangular PBW generating system for $A$. Write $\II_{n}'=\{\ell_{1}, \cdots , \ell_{m}\}$ with $\ell_{1}< \cdots < \ell_{m}$. Proposition \ref{ring-homo-properties} yields $\GKdim A =m$. For each $i=1,\ldots, m$, set $\II_{n,i}'= \{\ell_{1},\ldots, \ell_{i}\}$ and
		\[
		A_{i}:= \Span\{\hat{z}_{V} \mid V\in \PP(\II_{n,i}', {<}|_{\II_{n,i}'}) \}.
		\]
		Observe that $A_{m}=A$. By Proposition \ref{PBW-imply-IOE},
		\[
		\Bbbk=:A_{0}\subset A_{1} \subset \cdots \subset A_{m} =A
		\]
		is an increasing chain of subalgebras of $A$, and $A_{i}$ is an Ore extension of $A_{i-1}$ for every $i=1,\ldots, m$.
		
		Now consider the case that $A$ is a right coideal subalgebra of $H$. Clearly,
		\[
		\supp_{\hat{\GG}}(\Delta_{H}(\hat{z}_{\ell_{i}})) \subseteq \PP(\II_{n}', {<}|_{\II_{n}'})\times \PP(\II_{n}, <), \quad i=1,\ldots, m.
		\]
		By Proposition \ref{proposition of replacement}, $\hat{\GG}$ is $(\Delta_{H},\dd)$-dominated. It follows that for each $i=1,\ldots, m$,
		\begin{eqnarray*}
			\Delta_{H}(\hat{z}_{\ell_{i}}) &\in& 1\otimes \hat{z}_{\ell_{i}} + \hat{z}_{\ell_{i}}\otimes 1 \\
			&& + \Span\{\hat{z}_{P}\otimes \hat{z}_{Q} \mid (P,Q)\in \PP_{+}(\II_{n}', {<}|_{\II_{n}'})\times \PP_{+}(\II_{n}, <),\ P,Q <_{\rr} \ell_{i}  \} \\
			&\subseteq &A_{i}\otimes H,
		\end{eqnarray*}
		which shows that $A_{i}$ is a right coideal subalgebra of $H$.
		
		For the case that $A$ is a left coideal subalgebra (resp.\ Hopf subalgebra) of $H$, a similar argument proves that each $A_{i}$ inherits the corresponding property.
	\end{proof}
	
	\begin{remark}\label{example-order-preserving}
		In the proof of Theorem \ref{Hopf-subalgebra-structure}, we require an order-preserving positive map $\dd\colon\II_{n}\to\mathbb{N}$ such that $\GG=(\{z_{i}\}_{i\in \II_{n}}, <)$ is $\dd$-triangular and $(\Delta_{H},\dd)$-dominated. Lemma \ref{triangular-construct} guarantees the existence of such a map. A natural candidate is the positive map $\dd_{H}\colon\Xi \to \mathbb{N}$ defined by
		\[
		\dd_{H}(i)=\min\{\, m \geq 1 \mid z_{i}\in H_{(m)}\,\}.
		\]
		Frequently, $\GG$ is $\dd_{H}$-triangular and $(\Delta_{H},\dd_{H})$-dominated. Nevertheless, $\dd_{H}$ need not be order-preserving. An explicit example is given by the umbrella Hopf algebra $H=\mathrm{UM}(2,2)$ constructed in \cite{HW26}.
		
		We follow the notation from \cite{HW26}. The Hopf algebra $\mathrm{UM}(2,2)$ is generated by eight elements
		\[
		x_{0},\,x_{1},\,x_{2},\,y_{1},\,y_{2},\,
		e_{11}-e_{22},\,e_{12},\,e_{21},
		\]
		subject to the relations
		\begin{align*}
			& [x_{0},-]=0,  \\
			& [x_{1},x_{2}]=x_{0},    \quad 		[y_{1},y_{2}]=\tfrac13 x_{0}^{3},\\
			& [x_{i},y_{j}]=0, && i,j=1,2, \\
			& [x_{i},e_{11}-e_{22}]=(\delta_{i1}-\delta_{i2})x_{i}, \quad 
			[x_{i},e_{12}]=\delta_{i1}x_{2}, \quad
			[x_{i},e_{21}]=\delta_{i2}x_{1}, && i=1,2, \\
			& [y_{i},e_{11}-e_{22}]=(\delta_{i1}-\delta_{i2})y_{i}, \quad 
			[y_{i},e_{12}]=\delta_{i1}y_{2}, \quad 
			[y_{i},e_{21}]=\delta_{i2}y_{1}, && i=1,2, \\
			& [e_{11}-e_{22},e_{12}]=2e_{12}, \quad 
			[e_{11}-e_{22},e_{21}]=-2e_{21}, \quad 
			[e_{12},e_{21}]=e_{11}-e_{22}.
		\end{align*}
		Here, $\delta_{ij}$ denotes the Kronecker symbol. The comultiplication is given by
		\begin{align*}
			\Delta_{H}(x_{i})&=x_{i}\otimes1+1\otimes x_{i}, &&  i=0,1,2,\\
			\Delta_{H}(y_{i})&=y_{i}\otimes1+1\otimes y_{i}+x_{0}\otimes x_{i}-x_{i}\otimes x_{0}, &&  i=1,2, \\
			\Delta_{H}(M)&=M\otimes1+1\otimes M, &&  M = e_{11}-e_{22},\, e_{12},\, e_{21},
		\end{align*}
		and the counit $\varepsilon_{H}$ maps each generator to $0$.
		It turns out that $H=\mathrm{UM}(2,2)$ is a connected Hopf algebra of GK-dimension $8$, whose primitive space equals
		\[
		P(H)=\Span\{ x_{0}, x_{1}, x_{2}, e_{11}-e_{22}, e_{12}, e_{21}  \}.
		\]
		As observed in \cite[Remark 4.6]{HW26}, $H$ is an IHOE of $\Bbbk$ of the form
		\[
		H= \Bbbk[x_0][x_1][x_2;\delta_1][y_1;\delta_2][y_2;\delta_3]
		[e_{11}-e_{22};\delta_4][e_{12};\delta_5][e_{21};\delta_6].
		\]
		Let $\GG=(\{z_{i}\}_{i\in \II_{8}}, <)$ be the associated PBW generating system of $H$ furnished by Proposition \ref{IHOE properties} (5). One readily checks that $\GG$ is $\dd_{H}$-triangular and $(\Delta_{H}, \dd_{H})$-dominated. However, $\dd_{H}(i)= 1$ for all $i\neq 4,5$, while $\dd_{H}(4) = \dd_{H}(5)=2$. Consequently, $\dd_{H}$ fails to be order-preserving.
	\end{remark}
	
	Combining Theorem \ref{Graded-imply-IHOE} and Theorem \ref{Hopf-subalgebra-structure}, we immediately obtain the following corollary.
	
	\begin{corollary}\label{coideal-structure-homo}
		Let $H$ be a connected $\Gamma$-graded Hopf algebra of finite GK-dimension. Let $A$ be a Hopf (resp.\ left coideal, resp.\ right coideal) subalgebra of $H$. Then there exists a chain
		\[
		\Bbbk =A_{0} \subset A_{1} \subset \cdots \subset A_{m} =A
		\]
		of Hopf (resp.\ left coideal, resp.\ right coideal) subalgebras of $H$
		such that $m=\GKdim A$, and $A_{i}$ is an Ore extension of $A_{i-1}$ for every $i=1,\ldots, m$.  \hfill $\Box$
	\end{corollary}
	
	It is worth noting that Corollary \ref{coideal-structure-homo} imposes no homogeneity assumption on $A$ inside $H$. For  homogeneous Hopf subalgebras and  homogeneous left/right coideal subalgebras, Li and the second-named author established stronger relative statements in \cite{LZ23,Zh24}, reproduced below.
	
	\begin{thm}(\cite[Theorem B]{LZ23}, \cite[Theorem B]{Zh24}) \label{coideal-structure-homo-relative}
		Let $H$ be a connected $\Gamma$-graded Hopf algebra. Let $B\subseteq A$ be homogeneous Hopf (resp.\ left coideal, resp.\ right coideal) subalgebras of $H$ with finite GK-dimensions. Then there exists a chain
		\[
		B =A_{0} \subset A_{1} \subset \cdots \subset A_{m} =A
		\]
		of homogeneous Hopf (resp.\ left coideal, resp.\ right coideal) subalgebras of $H$
		such that $m=\GKdim A-\GKdim B$ and $A_{i}$ is a graded Ore extension of $A_{i-1}$ for every $i=1,\ldots, m$. \hfill $\Box$
	\end{thm}
	
	\begin{remark}
		In view of Theorem \ref{Hopf-subalgebra-structure} and Theorem \ref{coideal-structure-homo-relative}, we are naturally led to formulate the following question: if $H$ is an IHOE of $\Bbbk$, must $H$ be an IHOE of every Hopf subalgebra of itself? A related open problem, stated in \cite[Question C]{LZ23}, asks whether every finite GK-dimensional connected Hopf algebra $H$ is an IHOE of its primitive part $U(P(H))$, the universal enveloping algebra of $P(H)$. The umbrella Hopf algebra $H=\mathrm{UM}(2,2)$, constructed in \cite{HW26} and revisited in Remark \ref{example-order-preserving}, provides a counterexample to both questions. Specifically, $H$ is an IHOE of $\Bbbk$, yet not an IHOE of its primitive part; see \cite[Corollary 4.5]{HW26}. Moreover, $H$ has GK-dimension $8$, which is minimal among all such known counterexamples.
	\end{remark}
	
	We say that a Hopf subalgebra $K$ of a Hopf algebra $H$ is \emph{normal} if $K^{+}H = HK^{+}$. This definition differs from the classical one in \cite[Definition 3.4.1]{Mo93}, which requires
	\[
	\sum h_{(1)} k S(h_{(2)}) \in K
	\]
	for all $h \in H$, $k \in K$. When $H$ is connected, and more concretely when $H$ is an IHOE of $\Bbbk$, $H$ is free as a left (and right) $K$-module by \cite[Corollary 3.4]{LZ23} or \cite[Theorem 4]{RA77}. Under this freeness assumption, the two definitions are equivalent by \cite[Lemma 3.4.2, Proposition 3.4.3]{Mo93}.
	
	The next example illustrates that a Hopf subalgebra of an IHOE of $\Bbbk$ need not be normal.
	
	\begin{example}
		Let $H=B(\lambda)$ denote the connected Hopf algebra of GK-dimension $3$ constructed in \cite{Zh13}. It is generated by $X,Y,Z$ subject to the relations
		\[
		[X,Y]=Y,\qquad [Z,X]=-Z+\lambda Y,\qquad [Z,Y]=\tfrac12 Y^{2}.
		\]
		The algebra admits an IHOE presentation
		\[
		B(\lambda)=\Bbbk[Y][X;\delta_{2}][Z;\sigma_{3},\delta_{3}],
		\]
		where $\delta_{2}(Y)=Y$, $\sigma_{3}(Y)=Y$, $\sigma_{3}(X)=X-1$, $\delta_{3}(Y)=\tfrac12 Y^{2}$, $\delta_{3}(X)=\lambda Y$.
		Set $A = \Bbbk[X]$. Since $X$ is primitive, $A$ is a Hopf subalgebra of $H$, and $A$ itself is an IHOE.
		
		We claim $HA_{+}\neq A_{+}H$, which implies $A$ is not normal in $H$. Suppose for contradiction that $HA_{+}=A_{+}H$.
		Take $X\in A_{+}$ and $Y\in H$. From the commutator relation $[X,Y]=Y$, we obtain $Y=XY-YX\in A_{+}H$. Hence $Y = Xf$ for some $f\in H$. It is known that $\{X^{i}Y^{j}Z^{k} \mid i,j,k\geq 0\}$ forms a basis of $B(\lambda)$. Write $f=\sum_{i,j,k\geq 0} a_{i,j,k} X^{i}Y^{j}Z^{k}$. Then
		\[
		Y=  X\biggl(\sum_{i,j,k\geq 0} a_{i,j,k} X^{i}Y^{j}Z^{k}\biggr)
		=\sum_{i,j,k\geq 0} a_{i,j,k} X^{i+1}Y^{j}Z^{k}.
		\]
		This contradicts the linear independence of the monomial basis $\{X^{i}Y^{j}Z^{k}\mid i,j,k\ge 0\}$.
	\end{example}
	
	\subsection{Quotients of IHOEs}
	
	This subsection investigates the stability of the IHOE property under Hopf quotients. The following theorem is another main result of this paper.
	
	\begin{thm}\label{Quotients-IHOE}
		Let $H$ be an IHOE of $\Bbbk$. For any Hopf ideal $\mathfrak{a}\subset H$, the quotient Hopf algebra $H/\mathfrak{a}$ is also an IHOE of $\Bbbk$.
	\end{thm}
	
	\begin{proof}
		We may assume $H\neq \Bbbk$. For convenience, write $\Delta$ and $\overline{\Delta}$ for the coproducts of $H$ and $\overline{H}:=H/\mathfrak{a}$, respectively. Consider $H$ as a right $\overline{H}$-comodule via the canonical map $H\to \overline{H}$. By the well-known Masuoka's theorem, the  coinvariant  subspace $A:=H^{\operatorname{co} \overline{H}}$  is a left coideal subalgebra of $H$ and $\mathfrak{a}=A^{+}H$; see \cite{Ma91} or \cite[Theorem 6.3.2]{HS20} for references.
		
		By Proposition \ref{IHOE properties}(5), $H$ admits a triangular, $\Delta_H$-dominated PBW generating system $\GG=(\{z_i\}_{i\in \II_{n}}, <)$ with $z_i\in H^+$ for all $i$ and some integer $n\ge 1$, where $<$ denotes the canonical total order on $\II_n$ specified by $1<\cdots <n$. By Lemma \ref{triangular-construct}, we may choose an order-preserving positive map $\dd\colon\II_{n}\to\mathbb{N}$ such that $\GG$ is $\dd$-triangular and $(\Delta_{H},\dd)$-dominated. Define
		\[
		\II_{n}'=\II_{n}(A,\GG,\mathfrak{d}):=\{\, i\in\II_{n}\mid (z_{i}+H[\GG,i,\mathfrak{d}]) \cap A \neq \emptyset\},
		\]
		and choose a $\mathfrak{d}$-thin replacement $\hat{\GG}=(\{\hat{z}_{i}\}_{i\in \II_{n}}, <)$ of $\GG$ such that $\hat{z}_{i} \in A$ for all $i\in \II_{n}'$.  By Proposition \ref{proposition of replacement}, $\hat{\GG}$ remains $\mathfrak{d}$-triangular and $(\Delta, \mathfrak{d})$-dominated, and its restriction
		\[
		\hat{\GG}|_{\II_{n}'}:= \bigl(\{\hat{z}_{i}\}_{i\in \II_{n}'},\, {<}|_{\II_{n}'}\bigr)
		\]
		constitutes a PBW generating system for $A$ thanks to Theorem \ref{PBW-sub-hereditary}. Let $$\II_{n}''=\II_{n}\setminus\II_{n}'$$ and set
		\[
		\overline{z}_{i} = \hat{z}_{i} + \mathfrak{a} \in \overline{H}, \quad i=1,\ldots, n.
		\]
		We claim that $\overline{\GG}=(\{\overline{z}_{i}\}_{i\in \II_{n}''},{<}|_{\II_{n}''})$
		is a PBW generating system of $\overline{H}$ which is triangular and $\overline{\Delta}$-dominated. Once this claim is verified, the assertion follows directly from Proposition \ref{PBW-imply-IHOE}.
		
		Note that necessarily $\hat{z}_{i} \in H^{+}$ for every $i\in \II_{n}$, so $\hat{z}_{i} \in A^{+}$ whenever $i\in \II_{n}'$. Thus $\overline{z}_{V}=0$ for any word $V$ containing an index from $\II_{n}'$. Hence
		\begin{eqnarray*}
			\overline{z}_{j}\overline{z}_{i} &\in& \Bbbk^{\times}\overline{z}_{i} \overline{z}_{j} + \Span\{\overline{z}_{V}\mid V\in \PP_{+}(\II_{n},<), ~  V<_{\rr,\lex} ij \} \\
			& = & \Bbbk^{\times} \overline{z}_{i} \overline{z}_{j} + \Span\{\overline{z}_{V}\mid V\in \PP_{+}(\II_{n}'',{<}|_{\II_{n}''}),  ~ V<_{\rr,\lex} ij \}
		\end{eqnarray*}
		for all $i,j \in \II_{n}''$ with $i< j$, and
		\begin{eqnarray*}
			\overline{\Delta}(\overline{z}_{i}) &\in&  1\otimes \overline{z}_{i} + \overline{z}_{i} \otimes 1 + \Span\{\overline{z}_{P}\otimes \overline{z}_{Q}\mid P,Q\in \PP_{+}(\II_{n}, <), ~ P, Q <_{\rr} i\} \\
			&=& 1\otimes \overline{z}_{i} + \overline{z}_{i} \otimes 1 + \Span\{\overline{z}_{P}\otimes \overline{z}_{Q}\mid P,Q\in \PP_{+}(\II_{n}'',{<}|_{\II_{n}''}), ~ P, Q <_{\rr} i\}
		\end{eqnarray*}
		for every $i\in \II_{n}''$.
		Consequently, the system $\overline{\GG}$ is triangular and $\overline{\Delta}$-dominated.
		
		It remains to prove that $\{\overline{z}_{V}\}_{V\in \PP(\II_{n}'',{<}|_{\II_{n}''})}$ forms a basis of $\overline{H}$. The spanning property is clear. To establish linear independence, we introduce a total order $\triangleleft$ on $\II_{n}$ defined as follows: for $i,j\in \II_{n}$,
		\begin{eqnarray*}
			i\triangleleft j
			& \Longleftrightarrow &
			\left\{
			\begin{aligned}
				& i\in \II_{n}' \text{ and } j\in \II_{n}'', \quad \text{or} \\
				& i < j, \quad \text{where either } i, j\in \II_{n}' \text{ or } i, j\in \II_{n}''.
			\end{aligned}\right .
		\end{eqnarray*}
		By Proposition \ref{change-of-order}, $\{\hat{z}_{V}\}_{V\in \PP(\II_{n}, \triangleleft)}$ is a basis of $H$. Take arbitrary $P\in \PP_{+}(\II_{n}', {<}|_{\II_{n}'})$ and  $Q\in \PP(\II_{n}, \triangleleft)$. We decompose $Q=Q_{1}Q_{2}$ with $Q_{1}\in \PP(\II_{n}',{<}|_{\II_{n}'}) $ and $Q_{2} \in \PP(\II_{n}'',{<}|_{\II_{n}''})$. By Lemma \ref{rearrangement},
		\begin{eqnarray*}
			\hat{z}_{P}\hat{z}_{Q} &=& \hat{z}_{PQ_{1}} \hat{z}_{Q_{2}} \\
			&\in & \bigl( \Span\big\{\hat{z}_{W}  \,\big|\, W \in \PP_{+}(\II_{n}', {<}|_{\II_{n}'}), ~  W\leq _{\rr,\lex} \Pi(P Q_{1}) \big\} \bigr) \cdot \hat{z}_{Q_{2}}  \\
			&\subseteq & \Span\{\hat{z}_{V} \mid V\in \PP_{+}(\II_{n}, \triangleleft) \setminus \PP_{+}(\II_{n}'', {<}|_{\II_{n}''})\}.
		\end{eqnarray*}
		Consequently,
		\[
		\mathfrak{a}=A^{+}H =\Span\{\hat{z}_{V} \mid V\in \PP_{+}(\II_{n}, \triangleleft) \setminus \PP_{+}(\II_{n}'', {<}|_{\II_{n}''})\}.
		\]
		Linear independence of $\{\overline{z}_{V}\}_{V\in \PP(\II_{n}'',{<}|_{\II_{n}''})}$ follows immediately.
	\end{proof}

	\begin{remark}\label{exact-sequence-Hopf}
		Recall from \cite{AD95} that a sequence of Hopf algebras
		\begin{eqnarray*}\label{exact-sequence}
			\Bbbk \longrightarrow K \xrightarrow{\iota} H \xrightarrow{\pi} J \longrightarrow \Bbbk
		\end{eqnarray*}
		is called \emph{exact} if $\iota$ is injective, $\pi$ is surjective,
		\[
		\ker\pi = \iota(K^{+}) H \quad \text{and} \quad \im\iota = H^{\operatorname{co}J}.
		\]
		Since $\iota$ is injective, we may identify $K$ with a Hopf subalgebra of $H$. By \cite[Proposition 1.2.3]{AD95}, $K$ is normal in $H$ and $J\cong H/K^{+}H$ as Hopf algebras for any such exact sequence. It is clear that if $H$ is connected, then both $K$ and $J$ are connected as well. Combining Theorem \ref{Hopf-subalgebra-structure} and Theorem \ref{Quotients-IHOE}, we deduce that if $H$ is an IHOE of $\Bbbk$, then both $K$ and $J$ are too. On the other hand, Hu proved in her PhD thesis \cite{Hu26} that if $K$ and $J$ are connected, then so is $H$. This raises the natural question whether an analogous closure property holds for the class of IHOEs of $\Bbbk$.
	\end{remark}
	
	\subsection{Minimal PBW generating systems}
	
	In this subsection we introduce minimal PBW generating systems for Hopf algebras and relate them to the computation of coradical filtrations.
	
	\begin{definition}
		Let $H$ be a  Hopf algebra. An ordered family $\GG=(\{z_{\xi}\}_{\xi\in \Xi},<)$ of elements in $H$ is called \emph{$\dd$-minimal}, where $\dd\colon \Xi\to \mathbb{N}$ is a positive map, if
		\[
		\bigl( z_{\xi} + H_{+}[\GG, \xi,\dd] \bigr) \cap H_{(\dd(\xi)-1)} = \emptyset, \quad \xi \in \Xi.
		\]
	\end{definition}
	
	\begin{prop}\label{coradical filtration}
		Let $H$ be a Hopf algebra. Let $\GG=(\{z_{\xi}\}_{\xi\in \Xi},<)$ be a finite PBW generating system of $H$. Assume $\GG$ is $\dd_{H}$-minimal, $\dd_{H}$-triangular and $(\Delta_{H},\dd_{H})$-dominated, where $\dd_{H}\colon\Xi \to \mathbb{N}$ is the positive map defined by $\dd_{H}(\xi)=\min\{\, m \geq 1 \mid z_{\xi}\in H_{(m)}\,\}$. Then
		\[
		H_{(m)} =\Span\{z_{V} \mid V\in \mathcal{P}(\Xi,<), ~ \dd_{H}(V)\leq m\}, \quad m\geq 0.
		\]
	\end{prop}
	
	The proof of this proposition relies on the following technical lemma.
	
	\begin{lemma}\label{primitive-element}
		Let $H=\bigoplus_{\gamma\in \Gamma}H_{\gamma}$ be a connected $\Gamma$-graded Hopf algebra. Let $\GG=(\{z_{\xi}\}_{\xi\in \Xi},<)$ be a homogeneous PBW generating system of $H$. Assume $\GG$ is $\dd$-triangular and $(\Delta_{H},\dd)$-dominated, where $\dd\colon \Xi\to \Gamma$ is the positive map given by $\xi\mapsto \deg(z_{\xi})$. If $f\in P(H)$ is a nonzero homogeneous primitive element of $H$, then there exists $\xi\in \Xi$ such that
		\[
		f\in \Bbbk^{\times} \cdot z_{\xi} + \Span\bigl\{z_{V}\,\big|\, V\in \mathcal{P}_{+}(\Xi,<), ~ \dd(V) =\dd(\xi), ~ V<_{\rr}\xi \bigr\}.
		\]
	\end{lemma}
	
	\begin{proof}
		Let $\deg(f)=\gamma$. Since the family $\{z_{V}\}_{V\in \mathcal{P}(\Xi, <)}$ forms a homogeneous basis of $H$, we write
		\[
		f = a_{1}z_{V_{1}} + \cdots + a_{s}z_{V_{s}}
		\]
		for some $a_{1},\ldots, a_{s}\in \Bbbk^{\times}$ and $V_{1},\ldots, V_{s}\in \mathcal{P}_{+}(\Xi,<)$ satisfying
		\[
		\dd(V_{1})=\cdots = \dd(V_{s}) =\gamma \quad \text{and} \quad V_{1} <_{\rr} \cdots <_{\rr} V_{s}.
		\]
		In particular, $V_{1} <_{\rr,\dd} \cdots <_{\rr,\dd} V_{s}$. By Proposition \ref{comultiplication},
		\begin{eqnarray*}
			\Delta_{H}(f) & \in & \sum_{i=1}^{s} \Big ( \sum_{W|V_{i}} a_{i} \binom{V_{i}}{W} z_{W}\otimes z_{V_{i}/W} \\
			&& + \Span\bigl\{z_{P}\otimes z_{Q}\,\big|\, (P,Q) \in \mathcal{P}_{+}(\Xi,<)^{2}, ~ \dd(PQ)= \gamma, ~ \Pi(PQ) <_{\rr} V_{i} \bigr\} \Big) \\
			& \subseteq & 1\otimes f+ f\otimes 1 +   \sum_{ \substack{W|V_{s} \\ W\neq \emptyset, W\neq V_{s}}} a_{s} \binom{V_{s}}{W} z_{W}\otimes z_{V_{s}/W} \\
			&& + \Span\bigl\{z_{P}\otimes z_{Q}\,\big|\, (P,Q) \in \mathcal{P}_{+}(\Xi,<)^{2}, ~ \dd(PQ)= \gamma, ~ \Pi(PQ) <_{\rr} V_{s} \bigr\}.
		\end{eqnarray*}
		It follows that $(W,V_{s}/W) \in \supp_{\GG}(\Delta(f))$ for any proper nonempty subword $W$ of $V_{s}$. But $\Delta_{H}(f) =1\otimes f+f\otimes 1$, so $V_{s}$ has length $1$, i.e., $V_{s} =\xi\in \Xi$.
	\end{proof}
	
	\begin{proof}[Proof of Proposition \ref{coradical filtration}]
		For convenience, set $\Delta=\Delta_{H}$, $\dd = \dd_{H}$ and
		\[
		H_{[m]}:=\Span\{z_{V}\mid V\in\mathcal{P}(\Xi,<),\ \dd(V)\le m\},\quad m\ge0.
		\]
		By convention, let $H_{[-1]} = H_{(-1)} = 0$. By Lemma \ref{rearrangement} and Proposition \ref{comultiplication}, the sequence $\mathfrak{F}:=(H_{[m]})_{m\geq 0}$ is a Hopf algebra filtration on $H$ satisfying
		\[
		H_{[0]}=H_{(0)} =\Bbbk \cdot 1_{H}, \quad H_{[m]}\subseteq H_{(m)} \quad \text{for all} \quad m\geq 1.
		\]
		Hence the associated graded algebra
		\[
		\gr_{\mathfrak{F}}(H):= \bigoplus_{m\geq 0}H_{[m]}/H_{[m-1]}
		\]
		is a connected $\mathbb{N}$-graded Hopf algebra, and the canonical map
		\begin{eqnarray*}
			\Phi \colon \gr_{\mathfrak{F}}(H) &\longrightarrow& \gr_{c}(H) \\
			x+ H_{[m-1]} &\mapsto& x+H_{(m-1)}
		\end{eqnarray*}
		is a homomorphism of connected $\mathbb{N}$-graded Hopf algebras. We first claim that $P(\gr_{\mathfrak{F}}(H)) = \gr_{\mathfrak{F}}(H)_{1}$. Then $\gr_{\mathfrak{F}}(H)$ is strictly graded and $\Phi|_{P(\gr_{\mathfrak{F}}(H))}$ is injective. Therefore $\Phi$ is injective by \cite[Corollary 1.3.11]{HS20} or \cite[Lemma 5.3.3]{Mo93}. Define
		\[
		\overline{z}_{\xi}:= z_{\xi} +H_{[\dd(\xi)-1]} \in H_{[\dd(\xi)]}/H_{[\dd(\xi)-1]} \subseteq \gr_{\mathfrak{F}}(H), \quad \xi \in \Xi.
		\]
		One readily verifies that $\overline{\GG}:=(\{\overline{z}_{\xi}\}_{\xi\in \Xi}, <)$ forms a PBW generating system of $\gr_{\mathfrak{F}}(H)$. In particular, $\gr_{\mathfrak{F}}(H)$ is generated by $\{\overline{z}_{\xi}\}_{\xi\in \Xi}$ and thus finitely generated. We obtain
		\[
		\GKdim \gr_{c}(H) = \GKdim H = \GKdim \gr_{\mathfrak{F}}(H) = \#(\Xi) < \infty.
		\]
		Here the first equality follows from \cite[Proposition 3.4]{ZSL20}, the second from \cite[Proposition 6.6]{KL00}, and the third from Lemma \ref{GKdimension}. Thus $\Phi$ is an isomorphism by \cite[Theorem A]{LZ23}. Consequently, $H_{(m)} = H_{[m]}$ for all $m\geq 0$, which completes the proof.
		
		It remains to verify $P(\gr_{\mathfrak{F}}(H)) = \gr_{\mathfrak{F}}(H)_{1}$. Suppose otherwise. Then there exists a nonzero homogeneous primitive element $0\neq f\in \gr_{\mathfrak{F}}(H)_{m}$ for some integer $m\geq 2$. One checks that $\overline{\GG}=(\{\overline{z}_{\xi}\}_{\xi\in \Xi}, <)$ is $\dd$-triangular and $(\overline{\Delta}, \dd)$-dominated, where $\overline{\Delta}$ denotes the comultiplication of $\gr_{\mathfrak{F}}(H)$. By Lemma \ref{primitive-element}, $f$ can be written as
		\[
		f= a_{0} \overline{z}_{\xi} + \sum_{i=1}^{s}a_{i}\overline{z}_{V_{i}},
		\]
		with $s\geq 0$, $a_{i} \in \Bbbk^{\times}$, $\xi\in \Xi$, $\dd(\xi)=m$, and each $V_{i} \in \mathcal{P}(\Xi)$ satisfying $\dd(V_{i})=m$ and $V_{i}<_{\rr} \xi$. Set
		\[
		F=z_{\xi} + \frac{1}{a_{0}}\sum_{i=1}^{s} a_{i} z_{V_{i}}.
		\]
		Since $f = \overline{F}:= F + H_{[m-1]} \in \gr_{\mathfrak{F}}(H)$ and $\overline{\Delta}(f) =1\otimes f+f\otimes 1$, we compute
		\begin{eqnarray*}
			\Delta(F) &=& 1\otimes F+F\otimes 1 + \Span\bigl\{z_{P}\otimes z_{Q} \,\big|\, (P,Q) \in \mathcal{P}_{+}(\Xi,<)^{2}, ~ \dd(PQ)\leq m-1 \bigr\} \\
			&\subseteq& H_{(0)}\otimes H + H\otimes H_{(m-2)}.
		\end{eqnarray*}
		This yields $F\in H_{(m-1)}$. Meanwhile, $F\in z_{\xi} + H_{+}[\GG, \xi, \dd]$. Therefore
		\[
		\bigl( z_{\xi} + H_{+}[\GG, \xi, \dd] \bigr) \cap H_{(m-1)} \neq \emptyset,
		\]
		contradicting the minimality hypothesis. Hence $P(\gr_{\mathfrak{F}}(H)) = \gr_{\mathfrak{F}}(H)_{1}$, as desired.
	\end{proof}
	
	The following result provides a criterion for a Hopf algebra to be an IHOE of its primitive part.
	
	\begin{corollary}\label{IHOE-criterion}
		Let $H$ be a Hopf algebra admitting a finite PBW generating system $\GG=(\{z_{\xi}\}_{\xi\in \Xi}, <)$ which is $\dd_{H}$-minimal, $\dd_{H}$-triangular and $(\Delta_{H},\dd_{H})$-dominated, where $\dd_{H}\colon\Xi \to \mathbb{N}$ is the positive map defined by $\dd_{H}(\xi)=\min\{\, m \geq 1 \mid z_{\xi}\in H_{(m)}\,\}$. If $\dd_{H}$ is order-preserving, then $H$ is an IHOE of its primitive part $U(P(H))$, the universal enveloping algebra of $P(H)$.
	\end{corollary}
	
	\begin{proof}
		Let $\Xi=\{\xi_{1},\ldots, \xi_{n}\}$ with $\xi_{1} < \cdots < \xi_{n}$. For each $i=1,\ldots, n$, define $\Xi_{i}= \{\xi_{1},\ldots, \xi_{i}\}$.  By Proposition \ref{PBW-imply-IHOE}, $H$ is an $n$-step IHOE of $\Bbbk$ with presentation
		\[
		H=\Bbbk[z_{\xi_{1}};\sigma_{1},\delta_{1}]\cdots [z_{\xi_{n}};\sigma_{n},\delta_{n}].
		\]
		Since $\dd_{H}$ is order-preserving, there exists an integer $\ell\in\{1,\ldots, n\}$ such that $\{\xi\in \Xi\mid \dd_{H}(\xi) =1\} = \Xi_{\ell}$. By Proposition \ref{coradical filtration}, $P(H) =\Span\{z_{\xi} \mid \xi\in \Xi_{\ell}\}$ and hence $U(P(H)) =H_{\ell}$. Therefore, $H=H_{\ell}[z_{\xi_{\ell+1}};\sigma_{\ell+1},\delta_{\ell+1}]\cdots [z_{\xi_{n}};\sigma_{n},\delta_{n}]$ is an  IHOE of $U(P(H))$.
	\end{proof}
	
	\begin{remark}
		Consider the umbrella Hopf algebra $H=\mathrm{UM}(2,2)$; see Remark \ref{example-order-preserving}. It is known that $H$ is not an IHOE of its primitive part. Consequently, $H$ cannot admit any finite PBW generating system $\GG$ that is $\dd_{H}$-minimal, $\dd_{H}$-triangular and $(\Delta_{H},\dd_{H})$-dominated, with $\dd_{H}$ order-preserving. Here $\dd_{H}$ denotes the positive map appearing in Corollary \ref{IHOE-criterion}.
	\end{remark}

	\section{Computations for coideal subalgebras of IHOEs}
	\label{sec:classification}
	
	We now exploit the structural theory of IHOEs developed in Section \ref{sec:IHOE-structure} to explicitly classify their coideal subalgebras. Drawing on the thin replacement machinery for PBW generating systems, we construct an algorithm for classifying left and  right coideal subalgebras. We then implement this algorithm on the $3$-GK-dimensional connected Hopf algebras $A(\lambda_{1},\lambda_{2},\alpha)$ and $B(\lambda)$, obtaining complete classifications of their right coideal subalgebras.

	\begin{lemma}\label{coideal-GK-dimension-one}
		Let $A$ be a left or right coideal subalgebra of a connected Hopf algebra $H$. If $\GKdim A = 1$, then $A$ is a Hopf subalgebra generated by a primitive element.
	\end{lemma}
	
	\begin{proof}
		Set $n_{0}:=\min \{n\geq 0 \mid A_{(n)}\neq \Bbbk\}$, where $A_{(n)} =A\cap H_{(n)}$. Choose $x\in A_{(n_{0})}\setminus \Bbbk$. Then
		\[
		\Delta_{H}(x) -\bigl(1\otimes x+ x\otimes 1\bigr) \in (A\otimes H) \cap (H_{(n_{0}-1)} \otimes H) = A_{(n_{0}-1)} \otimes H =\Bbbk \otimes H.
		\]
		By the counit axiom, $\Delta(x) =1\otimes x+x\otimes 1$, so $x$ is primitive. Since $\Bbbk[x] \subseteq A$ and $\GKdim \Bbbk[x] =\GKdim A =1$, we obtain $\Bbbk[x] = A$ by \cite[Theorem A]{LZ23}.
	\end{proof}
	
	\begin{definition}
		Let $H$ be a connected Hopf algebra of finite GK-dimension. A left or right coideal subalgebra $A$ of $H$ is called \emph{substantial} if $2\leq \GKdim A \leq \GKdim H-1$.
	\end{definition}
	
	\subsection{General algorithm}
	\label{subsec:algorithm}
	
	Let $H$ be a Hopf algebra equipped with an iterated Hopf Ore extension structure
	\[
	H = \Bbbk[X_{1};\sigma_{1},\delta_{1}][X_{2};\sigma_{2},\delta_{2}]\cdots
	[X_{n};\sigma_{n},\delta_{n}].
	\]
	Define $z_{i} = X_{i} -\varepsilon_{H}(X_{i})$ for $i=1,\ldots, n$. By Proposition \ref{IHOE properties}(5), the ordered family $\GG=(\{z_{i}\}_{i\in \II_{n}}, <)$ is a triangular and $\Delta_{H}$-dominated PBW generating system of $H$.
	
	By Lemma \ref{triangular-construct}, we may fix an order-preserving positive map $\dd\colon\II_{n} \to\mathbb{N}$ such that $\GG$ is $\dd$-triangular and $(\Delta_{H},\dd)$-dominated. For any right (resp.\ left) coideal subalgebra $A$ of $H$, define
	\[
	\II_{n}'=\II_{n}(A,\GG,\dd):=\{\, i\in\II_{n}\mid (z_{i}+H[\GG,i,\dd]) \cap A \neq \emptyset\},
	\]
	and select a $\dd$-thin replacement $\hat{\GG}=(\{\hat{z}_{i}\}_{i\in \II_{n}}, <)$ of $\GG$ satisfying $\hat{z}_{i} \in A$ for all $i\in \II_{n}'$. Explicitly, the $\dd$-thin replacement may be constructed so that $\hat z_{i} =z_{i}$ whenever $i\not\in \II_{n}'$, while
	\[
	\hat z_{i}=z_{i}+\sum_{U}\alpha_{i,U}z_{U},
	\]
	for $i\in\II_{n}'$, with $\alpha_{i,U}\in\Bbbk$ and $U$ ranging over words in
	\[
	\{U \in \PP_{+}(\II_{n},<) \mid \dd(U)\le \dd(i), ~ U<_{\rr}i\} = \{U \in \PP_{+}(\II_{i-1},<) \mid \dd(U)\le \dd(i)\}.
	\]
	By Proposition \ref{proposition of replacement}, $\hat{\GG}$ remains $\dd$-triangular and $(\Delta_{H}, \dd)$-dominated. Hence Proposition \ref{PBW-imply-IHOE} implies that $H$ admits an IHOE presentation
	\begin{eqnarray}\label{IHOE-thin-replacement}
		H=\Bbbk[\hat{z}_{1};\hat{\sigma}_{1},\hat{\delta}_{1}]\cdots [\hat{z}_{n};\hat{\sigma}_{n},\hat{\delta}_{n}].
	\end{eqnarray}
	Furthermore, Theorem \ref{PBW-sub-hereditary} and Proposition \ref{ring-homo-properties} ensure that
	\[
	\hat{\GG}|_{\II_{n}'}:= \bigl(\{\hat{z}_{i}\}_{i\in \II_{n}'}, {<}|_{\II_{n}'}\bigr)
	\]
	forms a $\dd|_{\II_{n}'}$-triangular PBW generating system of $A$, with $\#(\II_{n}') =\GKdim A$.
	
	Accordingly,  classifying all right (resp.\ left) coideal subalgebras of $H$ reduces to determining which subsets $\II_{n}'\subset\II_{n}$ and parameters $\{\alpha_{i,U}\}$ yield families $\{\hat z_{i}\}_{i\in\II_{n}'}$ satisfying the required structural conditions. This observation motivates the following algorithm.
	
	\medskip\noindent
	\textbf{General algorithm for  classifying  right coideal subalgebras of $H$.}
	\begin{enumerate}
		\item[1.] \textbf{Choose a suitable map $\dd$.}
		Select an order-preserving positive map $\dd\colon\II_{n}\longrightarrow\mathbb{N}$
		such that $\GG=(\{z_{i}\}_{i\in \II_{n}}, <)$ is $\dd$-triangular and $(\Delta_{H},\dd)$-dominated. The $\dd$-thin perturbation of each $z_{i}$ takes the form
		\[
		\hat z_{i}=z_{i}+\sum_{U}\alpha_{i,U}z_{U},
		\]
		where $i\in\II_{n}$, $\alpha_{i,U}\in\Bbbk$, and $U$ runs over words in $\PP_{+}(\II_{i-1},<)$ satisfying $\dd(U)\le \dd(i)$.
		
		\item[2.] \textbf{Pick a subset $\II_{n}'\subset \II_{n}$ with $1\leq \#(\II_{n}') \leq n-1$.}
		Set $\alpha_{i,U}=0$ for all $i\not\in \II_{n}'$. Only the coefficients $\{\alpha_{i,U}\}$ with $i\in \II_{n}'$ remain to be determined, so that
		\[
		A:=\Span\{\hat{z}_{V}\mid V\in \PP(\II_{n}', {<}|_{\II_{n}'})\}
		\]
		becomes a right coideal subalgebra of $H$, with $\hat{\GG}|_{\II_{n}'}:= (\{\hat{z}_{i}\}_{i\in \II_{n}'}, {<}|_{\II_{n}'})$ serving as its $\dd|_{\II_{n}'}$-triangular PBW generating system.
		
		\item[3.] \textbf{Impose the right coideal condition.}
		To satisfy the right coideal property together with the $(\Delta_{H},\dd)$-dominated condition, we require for every $i\in\II_{n}'$
		\begin{eqnarray*}
			\Delta(\hat z_{i}) &\in& 1\otimes\hat z_{i}+\hat z_{i}\otimes 1 \\
			&& +\Span\bigl\{\hat z_{P}\otimes\hat z_{Q}\,\big|\,
			P\in\PP(\II_{n}',{<}|_{\II_{n}'}),\; Q\in\PP(\II_{n},<),\;
			\dd(PQ)\le\dd(i),\; P, Q<_{\rr }i\bigr\}.
		\end{eqnarray*}
		Matching coefficients yields constraints $(\ast)$ on the parameters $\{\alpha_{i,U}\}$.
		
		\item[4.] \textbf{Impose the IHOE structure condition.}
		For $A$ to be a subalgebra compatible with the IHOE presentation \eqref{IHOE-thin-replacement}, we demand for all $i,j\in\II_{n}'$ with $i< j$
		\begin{eqnarray*}
			[\hat z_{j},\hat z_{i}] & \in & \bigl( \Span\bigl\{\hat z_{V}  \,\big|\, V \in \PP_{+}(\II_{n}',{<}|_{\II_{n}'}),~ \dd(V)\le \dd(i), ~ V<_{\rr }i\bigr\} \bigr) \cdot \hat z_{j}  \\
			&& + \Span\bigl\{\hat z_{V}\,\big|\, V \in \PP_{+}(\II_{n}',{<}|_{\II_{n}'}),~ \dd(V)\le \dd(ij), ~ V<_{\rr }j\bigr\}.
		\end{eqnarray*}
		By Proposition \ref{IHOE properties}(3) and the IHOE presentation \eqref{IHOE-thin-replacement}, whenever $\hat z_{i}$ is primitive we may strengthen the condition to
		\begin{eqnarray*}
			[\hat z_{j},\hat z_{i}] & \in &  \Bbbk \cdot \hat z_{j}  +  \Span\bigl\{\hat z_{V}\,\big|\, V \in \PP_{+}(\II_{n}',{<}|_{\II_{n}'}),~ \dd(V)\le \dd(ij), ~ V<_{\rr }j\bigr\}.
		\end{eqnarray*}
		Comparing coefficients produces further constraints $(\ast\ast)$ on $\{\alpha_{i,U}\}$.
		
		\item[5.] \textbf{Solve the system of equations $(\ast)$--$(\ast\ast)$.}
		Each solution $\{\alpha_{i,U}\}$ yields a right coideal subalgebra $
		A=\operatorname{Span}\{\hat{z}_{V}\mid V\in \PP(\II_{n}', {<}|_{\II_{n}'})\}$ with $\GKdim A= \#(\II_{n}')$
		
		\item[6.] \textbf{Output.}
		Let $\II_{n}'$ run over all subsets of $\II_{n}$ satisfying $1\leq \#(\II_{n}') \leq n-1$. Collecting all corresponding right coideal subalgebras, together with $\Bbbk$ and $H$ itself, we obtain a complete list of right coideal subalgebras of $H$. Among them, those corresponding to $2\leq \#(\II_{n}') \leq n-1$ are the substantial ones. 
	\end{enumerate}
	
	\begin{remark}
		By Lemma \ref{coideal-GK-dimension-one}, the parameter space of left (resp. right) coideal subalgebras of a Hopf algebra $H$ coincides with the projective space of its primitive space $P(H)$. Therefore, the nontrivial task is to classify the substantial  left (resp. right) coideal subalgebras.
	\end{remark}
	
	In the rest of this section, we apply this algorithm to classify all right coideal subalgebras for each member of the two representative families of noncocommutative connected Hopf algebras of GK-dimension $3$, originally classified by Zhuang~\cite{Zh13}. 
	
	\subsection{The Hopf algebra $A(\lambda_{1},\lambda_{2},\alpha)$} 
	\label{subsec:A-type}
	
	Recall from \cite{Zh13} that $A(\lambda_{1},\lambda_{2},\alpha)$ is the algebra generated by $X,Y,Z$ subject to the relations
	\begin{equation*}
		[X,Y]=0,\qquad [Z,X]=\lambda_{1} X-\alpha Y,\qquad [Z,Y]=\lambda_{2} Y,
	\end{equation*}
	where $\alpha,\lambda_{1},\lambda_{2}\in \Bbbk$ satisfy $\alpha=0$ whenever $\lambda_{1}\neq \lambda_{2}$, and $\alpha\in\{0,1\}$ when $\lambda_{1}=\lambda_{2}$. This is a connected Hopf algebra with comultiplication
	\begin{align*}
		\Delta(X)&=X\otimes1+1\otimes X,\\
		\Delta(Y)&=Y\otimes1+1\otimes Y,\\
		\Delta(Z)&=Z\otimes1+1\otimes Z+X\otimes Y-Y\otimes X.
	\end{align*}
	It admits an IHOE filtration
	\[
	\Bbbk :=A_{0}\subset  A_{1}=\Bbbk[Y] \subset  A_{2}=\Bbbk[Y][X] \subset
	A_{2}[Z;\delta_{Z}]=A(\lambda_{1},\lambda_{2},\alpha),
	\]
	where the derivation $\delta_{Z}\colon A_{2} \to A_{2}$ is defined by $\delta_{Z}(Y)=\lambda_{2} Y$ and $\delta_{Z}(X)=\lambda_{1} X-\alpha Y$.
	
	\begin{remark}
		It is clear that the primitive space of $A(\lambda_{1},\lambda_{2},\alpha)$ is spanned by $X,Y$. Consequently, the right coideal subalgebras of GK-dimension $1$ are parameterized by the projective line $\mathbb{P}^{1}$, where each $[a,b]\in \mathbb{P}^{1}$ corresponds to $\Bbbk[aX+bY]$.    
	\end{remark}
	
	We now apply the general algorithm to classify all substantial right coideal subalgebras of $A(\lambda_{1},\lambda_{2},\alpha)$. Consider the PBW generating system $\GG=(\{z_{i}\}_{i\in\II_{3}},<)$ with
	$z_{1}=Y$, $z_{2}=X$, $z_{3}=Z$ and ordering $1<2<3$.
	Let $\dd\colon\II_{3} \to \mathbb{N}$ be the map induced by the coradical filtration; then
	$\dd(1)=\dd(2)=1,\ \dd(3)=2$, so $\dd$ is order-preserving. One verifies that $\GG$ is $\dd$-triangular and $(\Delta,\dd)$-dominated. The $\dd$-thin replacements take the form
	\[
	\begin{aligned}
		\hat z_{1} &= z_{1},\\
		\hat z_{2} &= z_{2} + a z_{1},\\
		\hat z_{3} &= z_{3} + bz_{1} + cz_{2} + dz_{1}^{2} + ez_{1}z_{2} + fz_{2}^{2}
	\end{aligned}
	\]
	with $a,b,c,d,e,f\in\Bbbk$. Direct computation yields
	\[
	\begin{aligned}
		\Delta(\hat z_{1}) &= 1\otimes\hat z_{1} + \hat z_{1}\otimes1,\\
		\Delta(\hat z_{2}) &= 1\otimes\hat z_{2} + \hat z_{2}\otimes1,\\
		\Delta(\hat z_{3}) &= 1\otimes\hat z_{3} + \hat z_{3}\otimes1
		+ (e+1)(z_{2}\otimes z_{1}) + (e-1)(z_{1}\otimes z_{2}) \\
		&\qquad + 2d(z_{1}\otimes z_{1}) + 2f(z_{2}\otimes z_{2}).
	\end{aligned}
	\]
	Hence $\hat z_{1},\hat z_{2}$ are primitive, while $\hat z_{3}$ is not. Implementing Steps 2–5 of the general algorithm gives the following families.
	\begin{enumerate}
		\item $\II_{3}'=\{1,2\}$. We may set $a=0$. Since $\hat z_{1} =z_{1},\hat z_{2}=z_{2}$ are primitive and
		\[
		[\hat z_{2},\hat z_{1}]=[z_{2}, z_{1}]=0,
		\]
		the corresponding right coideal subalgebra is
		\[
		A= \Bbbk[z_{1}][z_{2}].
		\]
		
		\item $\II_{3}'=\{1,3\}$. We may set $b=d=0$. The right coideal condition $\Delta(\hat z_{3})\in A\otimes H$ forces the terms $2f(z_{2}\otimes z_{2})$ and $(e+1)(z_{2}\otimes z_{1})$ to vanish, so $f=0$ and $e=-1$. Therefore
		\[
		\hat z_{3} = z_{3}  + cz_{2}  - z_{1}z_{2},\qquad c\in\Bbbk.
		\]
		Moreover,
		\[
		[\hat z_{3},\hat z_{1}] = [z_{3}  + cz_{2} - z_{1}z_{2},\,z_{1}]
		= [z_{3},z_{1}] = \lambda_{2} z_{1} =\lambda_{2} \hat z_{1}.
		\]
		The resulting right coideal subalgebras are
		\[
		A =\Bbbk[z_{1}][z_{3}+cz_{2}-z_{1}z_{2};\delta], \quad c\in \Bbbk,
		\]
		where the derivation $\delta\colon \Bbbk[z_{1}] \to \Bbbk[z_{1}]$ satisfies $\delta(z_{1})=\lambda_{2} z_{1}$.
		
		\item $\II_{3}'=\{2,3\}$. We may set $c=f=0$. Rewrite $\Delta(\hat z_{3})$ as
		\[
		\begin{aligned}
			\Delta(\hat z_{3})
			&=1\otimes\hat z_{3}+\hat z_{3}\otimes1+ (e+1)(z_{2}\otimes z_{1})+(e-1)(z_{1}\otimes z_{2})
			+2d(z_{1}\otimes z_{1}) \\
			&= 1\otimes\hat z_{3}+\hat z_{3}\otimes1+\bigl[(e+1)z_{2}+2d z_{1}\bigr]\otimes z_{1}
			+  (e-1)z_{1}\otimes z_{2}.
		\end{aligned}
		\]
		The right coideal and $(\Delta,\dd)$-dominated conditions require both factors $(e+1)z_{2}+2d z_{1}$ and $(e-1)z_{1}$ to lie in
		$\Bbbk[\hat z_{2}]=\Bbbk[z_{2}+a z_{1}]$. This yields
		\[
		a=d \quad \text{and}\quad e=1.
		\]
		We further compute
		\[
		[\hat z_{3},\hat z_{2}] =[z_{3}+bz_{1}+az_{1}^{2} +z_{1}z_{2}, z_{2}+az_{1}]
		= (\lambda_{1} z_{2} - \alpha z_{1}) + a\lambda_{2} z_{1}
		= \lambda_{1} z_{2} + (a\lambda_{2}-\alpha)z_{1}.
		\]
		The IHOE structure condition demands this commutator belong to
		$\Bbbk[\hat z_{2}]$, which imposes
		\[
		\alpha = a(\lambda_{2}-\lambda_{1}).
		\]
		We split into cases based on the parameters of $A(\lambda_{1},\lambda_{2},\alpha)$.
		\begin{itemize}
			\item $A(\lambda_{1},\lambda_{2},0)$ with $\lambda_{1} \neq \lambda_{2}$. The constraint reduces to $a(\lambda_{2}-\lambda_{1})=0$, forcing $a=0$. The resulting right coideal subalgebras are
			\[
			A=\Bbbk[z_{2}][z_{3}+bz_{1}+z_{1}z_{2};\delta], \qquad b\in\Bbbk,
			\]
			where the derivation $\delta\colon\Bbbk[z_{2}]\rightarrow \Bbbk[z_{2}]$ satisfies $\delta(z_{2})=\lambda_{1}z_{2}$.
			\item $A(\lambda_{1},\lambda_{2},0)$ with $\lambda_{1} =\lambda_{2}$. The corresponding right coideal subalgebras are
			\[
			A=\Bbbk[z_{2}+az_{1}][z_{3}+bz_{1}+az_{1}^{2}+z_{1}z_{2};\delta], \quad a,b \in \Bbbk,
			\]
			where the derivation $\delta\colon\Bbbk[z_{2}+az_{1}]\rightarrow \Bbbk[z_{2}+az_{1}]$ satisfies $\delta(z_{2}+az_{1})=\lambda_{1}(z_{2}+az_{1})$.
			\item $A(\lambda_{1},\lambda_{2},1)$ with $\lambda_{1} = \lambda_{2}$. No right coideal subalgebras arise from this family.
		\end{itemize}
	\end{enumerate}
	
	\begin{table}[H] 
		\centering
		\caption{Substantial right coideal subalgebras of $A(\lambda_{1},\lambda_{2},\alpha)$} 
		\label{table1}
		\renewcommand{\arraystretch}{1.5}
		\begin{tabular}{|>{\centering\arraybackslash}m{2.6cm}|c|>{\raggedright\arraybackslash}m{5.1cm}|>{\centering\arraybackslash}m{1.8cm}|}
			\hline
			Family & $\II_{3}'$ & Right Coideal Subalgebras & Parameters \\
			\hline
			\multirow{3}*[-13pt]{\shortstack{$A(\lambda_{1},\lambda_{2},0)$\\ $\lambda_{1}\neq\lambda_{2}$}}
			& $\{1,2\}$ & $\Bbbk[z_{1}][z_{2}]$ & \\
			\cline{2-4}
			& $\{1,3\}$ & $\Bbbk[z_{1}][z_{3}+c z_{2}-z_{1}z_{2}; \delta]$, \newline $\delta(z_{1})=\lambda_{2} z_{1}$ & $c\in\Bbbk$ \\
			\cline{2-4}
			& $\{2,3\}$ & $\Bbbk[z_{2}][z_{3}+b z_{1}+z_{1}z_{2}; \delta]$, \newline $\delta(z_{2})=\lambda_{1}z_{2}$ & $b\in\Bbbk$ \\
			\hline
			\multirow{3}*[-12pt]{\shortstack{$A(\lambda_{1},\lambda_{2},0)$\\ $\lambda_{1}=\lambda_{2}$}}
			& $\{1,2\}$ & $\Bbbk[z_{1}][z_{2}]$ & \\
			\cline{2-4}
			& $\{1,3\}$ & $\Bbbk[z_{1}][z_{3}+c z_{2}-z_{1}z_{2};\delta]$, \newline $\delta(z_{1})=\lambda_{2}z_{1}$ & $c\in\Bbbk$ \\
			\cline{2-4}
			& $\{2,3\}$ & $\Bbbk[z_{2}+az_{1}][z_{3}+bz_{1}+az_{1}^{2}+z_{1}z_{2};\delta]$, \newline $\delta(z_{2}+az_{1})=\lambda_{1}(z_{2}+az_{1})$ & $a,b\in \Bbbk$ \\
			\hline
			\multirow{2}*[-7pt]{\shortstack{$A(\lambda_{1},\lambda_{2},1)$\\ $\lambda_{1}=\lambda_{2}$}}
			& $\{1,2\}$ & $\Bbbk[z_{1}][z_{2}]$ & \\
			\cline{2-4}
			& $\{1,3\}$ & $\Bbbk[z_{1}][z_{3}+c z_{2}-z_{1}z_{2};\delta]$, \newline $\delta(z_{1})=\lambda_{2}z_{1}$ & $c\in\Bbbk$ \\
			\hline
		\end{tabular}
	\end{table}
	
	\subsection{The Hopf algebra $B(\lambda)$}
	\label{subsec:B-type}
	
	Recall from \cite{Zh13} that $B(\lambda)$, where $\lambda\in \Bbbk$, is the algebra generated by $X,Y,Z$ subject to
	\begin{equation*}
		[X,Y]=Y,\qquad [Z,X]=-Z+\lambda Y,\qquad [Z,Y]=\tfrac12 Y^{2}.
	\end{equation*}
	It is a connected Hopf algebra with comultiplication
	\begin{align*}
		\Delta(Y)&=Y\otimes1+1\otimes Y,\\
		\Delta(X)&=X\otimes1+1\otimes X,\\
		\Delta(Z)&=Z\otimes1+1\otimes Z+X\otimes Y.
	\end{align*}
	Moreover, it admits an IHOE filtration
	\[
	\Bbbk:=B_{0}\subset B_{1}=\Bbbk[Y]\subset B_{2}=
	B_{1}[X;\delta_{2}]\subset
	B_{2}[Z;\sigma_{3},\delta_{3}]=B(\lambda),
	\]
	where the derivation $\delta_{2}\colon B_{1} \to B_{1}$ satisfies
	$\delta_{2}(Y)=Y$; the automorphism $\sigma_{3}\colon B_{2}\to B_{2}$ is given by
	$\sigma_{3}(Y)=Y$, $\sigma_{3}(X)=X-1$; and the $\sigma_{3}$-derivation $\delta_{3}\colon B_{2} \to B_{2}$ satisfies
	$\delta_{3}(Y)=\tfrac12 Y^{2}$, $\delta_{3}(X)=\lambda Y$.
	
	\begin{remark}
		It is clear that the primitive space of $B(\lambda)$ is spanned by $X,Y$. Consequently, the right coideal subalgebras of GK-dimension $1$ are parameterized by the projective line $\mathbb{P}^{1}$, where each $[a,b]\in \mathbb{P}^{1}$ corresponds to $\Bbbk[aX+bY]$.    
	\end{remark}
	
	We now apply the general algorithm to classify all substantial right coideal subalgebras of $B(\lambda)$. Consider the PBW generating system $\GG=(\{z_{i}\}_{i\in\II_{3}},<)$ with
	$z_{1}=Y$, $z_{2}=X$, $z_{3}=Z$ and ordering $1<2<3$.
	Let $\dd\colon\II_{3} \to \mathbb{N}$ be the map induced by the coradical filtration, so $\dd(1)=\dd(2)=1,\ \dd(3)=2$. The map $\dd$ is order-preserving, and one verifies that $\GG$ is $\dd$-triangular and $(\Delta,\dd)$-dominated.
	The $\dd$-thin replacements take the form
	\[
	\begin{aligned}
		\hat z_{1} &= z_{1},\\
		\hat z_{2} &= z_{2} + a z_{1},\\
		\hat z_{3} &= z_{3} + bz_{1} + cz_{2} + dz_{1}^{2} + ez_{1}z_{2} + fz_{2}^{2}
	\end{aligned}
	\]
	with $a,b,c,d,e,f \in \Bbbk$.
	Computing coproducts yields
	\[
	\begin{aligned}
		\Delta(\hat z_{1}) &= 1\otimes\hat z_{1} + \hat z_{1}\otimes1,\\
		\Delta(\hat z_{2}) &= 1\otimes\hat z_{2} + \hat z_{2}\otimes1,\\
		\Delta(\hat z_{3}) &= 1\otimes\hat z_{3} + \hat z_{3}\otimes1
		+ 2d(z_{1}\otimes z_{1}) + 2f(z_{2}\otimes z_{2}) \\
		&\qquad + (1+e)(z_{2}\otimes z_{1}) + e(z_{1}\otimes z_{2}).
	\end{aligned}
	\]
	Thus $\hat z_{1},\hat z_{2}$ are primitive, while $\hat z_{3}$ is not.
	Implementing Steps 2–5 of the general algorithm yields the following families.
	\begin{enumerate}
		\item $\II_{3}'=\{1,2\}$. We may set $a=0$. Since $\hat z_{1} =z_{1}, \hat z_{2} =z_{2}$ are primitive and
		\[
		[\hat z_{2},\hat z_{1}]=[z_{2},z_{1}]=z_{1}=\hat z_{1},
		\]
		the corresponding right coideal subalgebra is
		\[
		A=\Bbbk[z_{1}][z_{2};\delta],
		\]
		where the derivation $\delta\colon \Bbbk[z_{1}] \to \Bbbk[z_{1}]$ satisfies $\delta(z_{1})=z_{1}$.
		
		\item $\II_{3}'=\{1,3\}$. We may set $b=d=0$.
		The right coideal condition $\Delta(\hat z_{3})\in A\otimes H$ forces
		the terms $2f(z_{2}\otimes z_{2})$ and $(1+e)(z_{2}\otimes z_{1})$ to vanish,
		so $f=0$ and $e=-1$. Hence
		\[
		\hat z_{3} = z_{3}  + cz_{2} - z_{1} z_{2},\qquad c\in\Bbbk.
		\]
		Moreover,
		\[
		[\hat z_{3},\hat z_{1}] = [z_{3}  + cz_{2}  - z_{1}z_{2},\,z_{1}]
		= c z_{1} - \tfrac12 z_{1}^{2} = c \hat z_{1} - \tfrac12 \hat z_{1}^{2}.
		\]
		The resulting right coideal subalgebras are
		\[
		A = \Bbbk[z_{1}][z_{3}+cz_{2}-z_{1}z_{2};\delta], \quad c\in \Bbbk,
		\]
		where the derivation $\delta\colon \Bbbk[z_{1}]\to \Bbbk[z_{1}]$ is defined by $\delta(z_{1})=c z_{1}-\frac12 z_{1}^{2}$.
		
		\item $\II_{3}'=\{2,3\}$. We may set $c=f=0$.
		Rewrite $\Delta(\hat z_{3})$ as
		\[
		\begin{aligned}
			\Delta(\hat z_{3})
			&= 1\otimes\hat z_{3} + \hat z_{3}\otimes1
			+ 2d(z_{1}\otimes z_{1})
			+ (1+e)(z_{2}\otimes z_{1}) + e(z_{1}\otimes z_{2})\\
			&= 1\otimes\hat z_{3} + \hat z_{3}\otimes1
			+ \bigl[(1+e)z_{2}+2d z_{1}\bigr]\otimes z_{1}
			+ e z_{1}\otimes z_{2}.
		\end{aligned}
		\]
		The right coideal and $(\Delta,\dd)$-dominated conditions require both factors
		$(1+e)z_{2}+2d z_{1}$ and $e z_{1}$ to lie in
		$\Bbbk[\hat z_{2}]=\Bbbk[z_{2}+a z_{1}]$. This gives
		\[
		a=2d \quad \text{and}\quad e = 0.
		\]
		We compute the commutator, using $[z_{3},z_{2}]= -z_{3}+\lambda z_{1}$, $[z_{2},z_{1}]=z_{1}$,
		$[z_{3},z_{1}]=\frac12 z_{1}^{2}$:
		\[
		\begin{aligned}
			[\hat z_{3},\hat z_{2}]
			= \bigl[z_{3}+ bz_{2} +\tfrac{a}{2} z_{1}^{2},\,z_{2}+a z_{1}\bigr] = -\hat z_{3} + \lambda z_{1}.
		\end{aligned}
		\]
		The IHOE structure condition requires this commutator to belong to $\Bbbk \hat z_{3}+\Bbbk[\hat z_{2}]$. This enforces
		\[
		\lambda =0.
		\]
		For $\lambda=0$, the corresponding right coideal subalgebras are
		\[
		A = \Bbbk[z_{2}+ az_{1}][z_{3} + b z_{1} +\tfrac{a}{2} z_{1}^{2}; \sigma], \quad a,b\in \Bbbk,
		\]
		where the automorphism $\sigma\colon \Bbbk[z_{2}+ az_{1}]\to  \Bbbk[z_{2}+ az_{1}]$ satisfies $\sigma(z_{2}+ az_{1}) = z_{2}+ az_{1}-1$.
	\end{enumerate}
	
	\begin{table}[H]
		\centering
		\caption{Substantial right coideal subalgebras of $B(\lambda)$}
		\label{table3}
		\renewcommand{\arraystretch}{1.5}
		\begin{tabular}{|>{\centering\arraybackslash}m{2.6cm}|c|>{\raggedright\arraybackslash}m{5cm}|>{\centering\arraybackslash}m{1.6cm}|}
			\hline
			Family & $\II_{3}'$ & Right Coideal Subalgebras & Parameters \\
			\hline
			\multirow{2}*[-7pt]{\shortstack{$B(\lambda)$\\ $\lambda\neq0$}}
			& $\{1,2\}$ & $\Bbbk[z_{1}][z_{2}; \delta]$, \newline $\delta(z_{1})=z_{1}$ & \\
			\cline{2-4}
			& $\{1,3\}$ & $\Bbbk[z_{1}][z_{3}+c z_{2}-z_{1}z_{2}; \delta]$, \newline $\delta(z_{1})=c z_{1}-\tfrac12 z_{1}^{2}$ & $c\in\Bbbk$ \\
			\hline
			\multirow{3}*[-12pt]{$B(0)$}
			& $\{1,2\}$ & $\Bbbk[z_{1}][z_{2}; \delta]$, \newline $\delta(z_{1})=z_{1}$ & \\
			\cline{2-4}
			& $\{1,3\}$ & $\Bbbk[z_{1}][z_{3}+c z_{2}-z_{1}z_{2}; \delta]$, \newline $\delta(z_{1})=c z_{1}-\tfrac12 z_{1}^{2}$ & $c\in\Bbbk$ \\
			\cline{2-4}
			& $\{2,3\}$ & $\Bbbk[z_{2}+ az_{1}][z_{3} + b z_{1} +\tfrac{a}{2} z_{1}^{2}; \sigma]$, \newline $\sigma(z_{2}+a z_{1})=z_{2}+a z_{1}-1$ & $a,b\in \Bbbk$ \\
			\hline
		\end{tabular}
	\end{table}
	
	\subsection{Some observations}
	
	\begin{remark}
		Brown and Gilmartin classified the right coideal subalgebras of $B(\lambda)$; see \cite[Proposition 3.7]{BG16}, though several details are omitted in their work. However, we discover an additional family of right coideal subalgebras with GK-dimension $2$ in the case $\lambda=0$, given by
		\[
		\Bbbk[z_{2}+az_{1}]\bigl[z_{3}+bz_{1}+\tfrac{a}{2}z_{1}^{2};\sigma\bigr], \quad a,b\in \Bbbk.
		\]
		This family corresponds to $\II_{3}'=\{2,3\}$.
	\end{remark}
	
	\begin{remark}
		From our classification results, one observes that each member of the two families $A(\lambda_{1},\lambda_{2},\alpha)$ and $B(\lambda)$ admits exactly one substantial Hopf subalgebra, namely the universal enveloping algebra of its primitive space, corresponding to the case $\II_{3}'=\{1,2\}$. By contrast, each such Hopf algebra possesses infinitely many substantial right coideal subalgebras. This supports the viewpoint that coideal subalgebras should be regarded as intermediate substructures beyond Hopf subalgebras in Hopf algebra theory, a perspective already adopted for quantum groups and mentioned in the paragraph preceding Theorem \ref{main-theorem-2} in the Introduction.
	\end{remark}
	
	\begin{remark}
		The classification reveals that the parameter spaces of substantial right coideal subalgebras of $A(\lambda_{1},\lambda_{2},0)$ (resp. $B(\lambda)$) behave drastically differently when $\lambda_{1}\neq \lambda_{2}$ versus $\lambda_{1}=\lambda_{2}$ (resp. $\lambda\neq 0$ versus $\lambda=0$). We expect a conceptual explanation for this phenomenon.
	\end{remark}

	\section*{Acknowledgments}
	The idea of this paper originated when the second-named author prepared a talk for the IAMS-BIRS workshop “Poisson Geometry and Artin–Schelter Regular Algebras”, held in Hangzhou in October 2024. The second-named author would like to thank IAMS-BIRS and the workshop organizers for their hospitality. Part of this work was carried out during his visit to the Shanghai Center of Mathematical Sciences (SCMS) in autumn 2025, and he is grateful to Professor Quanshui Wu for his warm reception. The authors thank Dr. Mengying Hu of SCMS for pointing out a critical gap in an early version of this manuscript. This research was supported by the National Natural Science Foundation of China (NSFC) under Grant No. 12371039.

\Addresses

\end{document}